\documentclass[10pt,reqno]{amsart}

\usepackage[utf8x]{inputenc}
\usepackage[english]{babel}
\usepackage{amsmath,amsfonts,amssymb,amsthm,bm,shuffle}
\usepackage[T1]{fontenc}
\usepackage[math]{anttor}

\usepackage[top=3.5cm,bottom=3.5cm,left=3.6cm,right=3.6cm]{geometry}

\usepackage{xcolor}

\definecolor{ColBlack}{RGB}{0,0,0} 
\definecolor{ColWhite}{RGB}{255,255,255} 
\definecolor{Col1}{RGB}{133,6,6} 
\definecolor{Col2}{RGB}{198,8,0} 
\definecolor{Col3}{RGB}{174,74,52} 
\definecolor{Col4}{RGB}{103,113,121} 
\definecolor{Col5}{RGB}{90,94,107} 
\definecolor{Col6}{RGB}{70,63,50} 

\usepackage[hyperindex=true,frenchlinks=true,colorlinks=true,
citecolor=Col2,linkcolor=Col3,urlcolor=Col4,linktocpage,
pagebackref=true]{hyperref}

\usepackage{tikz}
\usetikzlibrary{calc}
\usetikzlibrary{shapes}
\usetikzlibrary{fit}
\usetikzlibrary{decorations.pathmorphing}
\usetikzlibrary{arrows.meta}

\usepackage{mathtools}
\usepackage{dsfont}
\usepackage{wasysym}
\usepackage{stmaryrd}
\usepackage{cite}
\usepackage{subfig}
\usepackage{multirow}
\usepackage{enumitem}
\usepackage{multicol}
\usepackage{cancel}

\allowdisplaybreaks

\numberwithin{equation}{subsection}

\makeatletter
\def\l@section{\@tocline{1}{3pt}{1pc}{5pc}{}}
\def\l@subsection{\@tocline{2}{2pt}{2pc}{5pc}{}}
\makeatother

\newtheorem{Theorem}{Theorem}[subsection]
\newtheorem{Proposition}[Theorem]{Proposition}
\newtheorem{Lemma}[Theorem]{Lemma}

\renewcommand{\leq}{\leqslant}
\renewcommand{\geq}{\geqslant}

\title[Quotients of the magmatic operad]
      {Quotients of the magmatic operad:\\
        lattice structures and convergent rewrite systems}
\date{\today}
\author{Cyrille Chenavier \and Christophe Cordero \and Samuele Giraudo}
\address{\scriptsize Université Paris-Est, LIGM (UMR $8049$), CNRS,
    ENPC, ESIEE Paris, UPEM, F-$77454$, Marne-la-Vallée, France}
\email{cyrille.chenavier@u-pem.fr}
\email{christophe.cordero@u-pem.fr}
\email{samuele.giraudo@u-pem.fr}

\newcommand{\Hide}[1]{\textcolor{Col4}{\tt [hidden]}}

\newcommand{\Def}[1]{\textcolor{Col3}{\em #1}}

\tikzstyle{Centering}=[{baseline={([yshift=-0.5ex]current
    bounding box.center)}}]

\newcommand{\N}{\mathbb{N}}

\newcommand{\K}{\mathbb{K}}

\newcommand{\Angle}[1]{\left\langle#1\right\rangle}

\newcommand{\Cca}{\mathcal{C}}
\newcommand{\Oca}{\mathcal{O}}
\newcommand{\Afr}{\mathfrak{a}}
\newcommand{\Rfr}{\mathfrak{r}}
\newcommand{\Sfr}{\mathfrak{s}}
\newcommand{\Tfr}{\mathfrak{t}}
\newcommand{\Zfr}{\mathfrak{z}}
\newcommand{\BranchingTrees}{\mathfrak{P}}

\newcommand{\LambdaB}{\bm{\lambda}}
\newcommand{\MuB}{\bm{\mu}}

\newcommand{\Zero}{\mathtt{0}}
\newcommand{\Two}{\mathtt{2}}
\newcommand{\HilbertSeries}{\mathcal{H}}
\newcommand{\GeneratingSet}{\mathfrak{G}}

\newcommand{\Mag}{\mathbf{Mag}}
\newcommand{\TwoMag}{2\Mag}
\newcommand{\KMag}{\K \Angle{\Mag}}
\newcommand{\KAs}{\K \Angle{\As}}
\newcommand{\QMag}{\mathcal{Q}\left(\KMag\right)}
\newcommand{\IMag}{\mathcal{I}\left(\KMag\right)}
\newcommand{\As}{\mathbf{As}}
\newcommand{\AAs}{\mathbf{AAs}}
\newcommand{\TwoNil}{2\mathbf{Nil}}
\newcommand{\RC}[1]{\mathbf{RC}^{\left(#1\right)}}
\newcommand{\KRC}[1]{\K\langle\RC{#1}\rangle}
\newcommand{\TwoAs}{2\As}
\newcommand{\CAsAll}{\mathbf{CAs}}
\newcommand{\CAs}[1]{\CAsAll^{\left(#1\right)}}
\newcommand{\KCAs}[1]{\K\langle\CAs{#1}\rangle}
\newcommand{\KCAsAll}{\K\langle\CAsAll\rangle}

\newcommand{\PrefixWord}{\mathrm{p}}
\newcommand{\Deg}{\mathrm{deg}}
\newcommand{\LComb}[1]{\mathfrak{c}^{\left(#1\right)}_{
\begin{tikzpicture}[xscale=.2,yscale=.17,Centering]
    \draw[Edge](0,0)--(-1,-1);
\end{tikzpicture}
}}
\newcommand{\RComb}[1]{\mathfrak{c}^{\left(#1\right)}_{
\begin{tikzpicture}[xscale=.2,yscale=.17,Centering]
    \draw[Edge](0,0)--(1,-1);
\end{tikzpicture}
}}
\newcommand{\LRank}{\mathrm{lr}}

\newcommand{\Catalan}{\mathrm{cat}}
\newcommand{\OrdQMag}{\preceq_{\mathrm{i}}}
\newcommand{\InfQMag}{\wedge_{\mathrm{i}}}
\newcommand{\SupQMag}{\vee_{\mathrm{i}}}
\newcommand{\LatQMag}{\left(\QMag,\OrdQMag,\InfQMag,\SupQMag\right)}
\newcommand{\OrdCAs}{\preceq_{\mathrm{d}}}
\newcommand{\InfCAs}{\wedge_{\mathrm{d}}}
\newcommand{\SupCAs}{\vee_{\mathrm{d}}}
\newcommand{\LatCAs}{\left(\CAsAll,\OrdCAs,\InfCAs,\SupCAs\right)}
\newcommand{\LatKCAs}{\left(\KCAsAll,\OrdCAs,\InfCAs,\SupCAs\right)}
\newcommand{\Unit}{\mathds{1}}
\newcommand{\Leaf}{{
\begin{tikzpicture}[xscale=.2,yscale=.22,Centering]
    \draw[Edge](0,0)--(0,-1);
\end{tikzpicture}
}}
\newcommand{\Corolla}{\mathrm{c}}
\newcommand{\NormalForms}{\mathfrak{N}}
\newcommand{\Hom}{\mathrm{Hom}}
\newcommand{\Compositions}{\mathcal{C}}

\DeclareMathOperator{\Product}{\star}
\DeclareMathOperator{\Congr}{\equiv}
\DeclareMathOperator{\Rew}{\to}

\DeclareMathOperator{\RewContext}{\Rightarrow}
\DeclareMathOperator{\RewContextT}{\overset{+}{\RewContext}}
\DeclareMathOperator{\RewContextRT}{\overset{\ast}{\RewContext}}
\DeclareMathOperator{\RewContextRST}{\overset{\ast}{\Leftrightarrow}}
\DeclareMathOperator{\lcm}{lcm}

\newcommand{\CongrCAs}[1]{\Congr^{\left(#1\right)}}
\newcommand{\CongrRC}[1]{\Congr_{\left(#1\right)}}
\newcommand{\INiHil}{I_{\TwoNil}}
\newcommand{\IAs}{I_{\KAs}}
\newcommand{\IAAs}{I_{\AAs}}
\newcommand{\IRC}[1]{I_{\RC{#1}}}

\newcommand{\Lightning}[1]{
\mathfrak{c}^{\left(#1\right)}_{\begin{tikzpicture}[scale=.1, Centering]
    \draw[Edge](0,0)--(-1,-2);
    \draw[Edge](-1,-2)--(0,-2);
    \draw[Edge](0,-2)--(-1,-4);
\end{tikzpicture}}}

\tikzstyle{Node}=[circle,draw=Col1!80,fill=Col1!8,inner sep=1pt,
minimum size=2mm,thick,font=\scriptsize]
\tikzstyle{Edge}=[draw=Col2!80,cap=round,thick]
\tikzstyle{EdgeColorF}=[Edge,draw=Col6!90,ultra thick]
\tikzstyle{Leaf}=[rectangle,draw=ColBlack!70,fill=ColBlack!16,
inner sep=0pt,minimum size=1mm,thick]
\tikzstyle{NodeST}=[font=\footnotesize]
\tikzstyle{Injection}=[ColBlack!100,draw,
{>[scale=1.5,length=4,width=5]}-{>[scale=1.5,length=4,width=5]}]
\tikzstyle{Map}=[ColBlack!100,draw,-{>[scale=1.5,length=4,width=5]}]

\newcommand{\TreeA}{
\begin{tikzpicture}[xscale=.22,yscale=.23,Centering]
    \node(0)at(0.00,-5.25){};
    \node(2)at(2.00,-5.25){};
    \node(4)at(4.00,-3.50){};
    \node(6)at(6.00,-1.75){};
    \node[NodeST](1)at(1.00,-3.50){\begin{math}\Product\end{math}};
    \node[NodeST](3)at(3.00,-1.75){\begin{math}\Product\end{math}};
    \node[NodeST](5)at(5.00,0.00){\begin{math}\Product\end{math}};
    \draw[Edge](0)--(1);
    \draw[Edge](1)--(3);
    \draw[Edge](2)--(1);
    \draw[Edge](3)--(5);
    \draw[Edge](4)--(3);
    \draw[Edge](6)--(5);
    \node(r)at(5.00,1.31){};
    \draw[Edge](r)--(5);
\end{tikzpicture}}

\newcommand{\TreeB}{
\begin{tikzpicture}[xscale=.22,yscale=.23,Centering]
    \node(0)at(0.00,-3.50){};
    \node(2)at(2.00,-5.25){};
    \node(4)at(4.00,-5.25){};
    \node(6)at(6.00,-1.75){};
    \node[NodeST](1)at(1.00,-1.75){\begin{math}\Product\end{math}};
    \node[NodeST](3)at(3.00,-3.50){\begin{math}\Product\end{math}};
    \node[NodeST](5)at(5.00,0.00){\begin{math}\Product\end{math}};
    \draw[Edge](0)--(1);
    \draw[Edge](1)--(5);
    \draw[Edge](2)--(3);
    \draw[Edge](3)--(1);
    \draw[Edge](4)--(3);
    \draw[Edge](6)--(5);
    \node(r)at(5.00,1.31){};
    \draw[Edge](r)--(5);
\end{tikzpicture}}

\newcommand{\TreeC}{
\begin{tikzpicture}[xscale=.22,yscale=.18,Centering]
    \node(0)at(0.00,-4.67){};
    \node(2)at(2.00,-4.67){};
    \node(4)at(4.00,-4.67){};
    \node(6)at(6.00,-4.67){};
    \node[NodeST](1)at(1.00,-2.33){\begin{math}\Product\end{math}};
    \node[NodeST](3)at(3.00,0.00){\begin{math}\Product\end{math}};
    \node[NodeST](5)at(5.00,-2.33){\begin{math}\Product\end{math}};
    \draw[Edge](0)--(1);
    \draw[Edge](1)--(3);
    \draw[Edge](2)--(1);
    \draw[Edge](4)--(5);
    \draw[Edge](5)--(3);
    \draw[Edge](6)--(5);
    \node(r)at(3.00,1.75){};
    \draw[Edge](r)--(3);
\end{tikzpicture}}

\newcommand{\TreeD}{
\begin{tikzpicture}[xscale=.22,yscale=.23,Centering]
    \node(0)at(0.00,-1.75){};
    \node(2)at(2.00,-5.25){};
    \node(4)at(4.00,-5.25){};
    \node(6)at(6.00,-3.50){};
    \node[NodeST](1)at(1.00,0.00){\begin{math}\Product\end{math}};
    \node[NodeST](3)at(3.00,-3.50){\begin{math}\Product\end{math}};
    \node[NodeST](5)at(5.00,-1.75){\begin{math}\Product\end{math}};
    \draw[Edge](0)--(1);
    \draw[Edge](2)--(3);
    \draw[Edge](3)--(5);
    \draw[Edge](4)--(3);
    \draw[Edge](5)--(1);
    \draw[Edge](6)--(5);
    \node(r)at(1.00,1.31){};
    \draw[Edge](r)--(1);
\end{tikzpicture}}

\newcommand{\TreeE}{
\begin{tikzpicture}[xscale=.22,yscale=.23,Centering]
    \node(0)at(0.00,-1.75){};
    \node(2)at(2.00,-3.50){};
    \node(4)at(4.00,-5.25){};
    \node(6)at(6.00,-5.25){};
    \node[NodeST](1)at(1.00,0.00){\begin{math}\Product\end{math}};
    \node[NodeST](3)at(3.00,-1.75){\begin{math}\Product\end{math}};
    \node[NodeST](5)at(5.00,-3.50){\begin{math}\Product\end{math}};
    \draw[Edge](0)--(1);
    \draw[Edge](2)--(3);
    \draw[Edge](3)--(1);
    \draw[Edge](4)--(5);
    \draw[Edge](5)--(3);
    \draw[Edge](6)--(5);
    \node(r)at(1.00,1.31){};
    \draw[Edge](r)--(1);
\end{tikzpicture}}

\begin{document}

\begin{abstract}
    We study quotients of the magmatic operad, that is the free
    nonsymmetric operad over one binary generator. In the linear
    setting, we show that the set of these quotients admits a lattice
    structure and we show an analog of the Grassmann formula for the
    dimensions of these operads. In the nonlinear setting, we define
    comb associative operads, that are operads indexed by nonnegative
    integers generalizing the associative operad. We show that the set
    of comb associative operads admits a lattice structure, isomorphic
    to the lattice of nonnegative integers equipped with the division
    order. Driven by computer experimentations, we provide a finite
    convergent presentation for the comb associative operad in
    correspondence with~$3$. Finally, we study quotients of the magmatic
    operad by one cubic relation by expressing their Hilbert series and
    providing combinatorial realizations.
\end{abstract}

\maketitle

\begin{small}
\tableofcontents
\end{small}

\section*{Introduction}
Associative algebras are spaces endowed with a binary product $\Product$
satisfying among others the associativity law
\begin{equation}
    \left(x_1 \Product x_2\right) \Product x_3
    =
    x_1 \Product \left(x_2 \Product x_3\right).
\end{equation}
It is well-known that the associative algebras are representations of
the associative (nonsymmetric) operad $\As$. This operad can be seen as
the quotient of the magmatic operad $\Mag$ (the free operad of binary
trees on the binary generator~$\Product$) by the operad
congruence~$\Congr$ satisfying
\begin{equation} \label{equ:congruence_as}
    \begin{tikzpicture}[xscale=.24,yscale=.24,Centering]
        \node(0)at(0.00,-3.33){};
        \node(2)at(2.00,-3.33){};
        \node(4)at(4.00,-1.67){};
        \node[NodeST](1)at(1.00,-1.67)
            {\begin{math}\Product\end{math}};
        \node[NodeST](3)at(3.00,0.00)
            {\begin{math}\Product\end{math}};
        \draw[Edge](0)--(1);
        \draw[Edge](1)--(3);
        \draw[Edge](2)--(1);
        \draw[Edge](4)--(3);
        \node(r)at(3.00,1.5){};
        \draw[Edge](r)--(3);
    \end{tikzpicture}
    \Congr
    \begin{tikzpicture}[xscale=.24,yscale=.24,Centering]
        \node(0)at(0.00,-1.67){};
        \node(2)at(2.00,-3.33){};
        \node(4)at(4.00,-3.33){};
        \node[NodeST](1)at(1.00,0.00)
                {\begin{math}\Product\end{math}};
        \node[NodeST](3)at(3.00,-1.67)
                {\begin{math}\Product\end{math}};
        \draw[Edge](0)--(1);
        \draw[Edge](2)--(3);
        \draw[Edge](3)--(1);
        \draw[Edge](4)--(3);
        \node(r)at(1.00,1.5){};
        \draw[Edge](r)--(1);
    \end{tikzpicture}\,.
\end{equation}
These two binary trees are the syntax trees of the expressions appearing
in the above associativity law.
\medbreak

In a more combinatorial context and regardless of the theory of operads,
the Tamari order is a partial order on the set of the binary trees
having a fixed number of internal nodes $\gamma$. This order is
generated by the covering relation consisting in rewriting a tree $\Tfr$
into a tree $\Tfr'$ by replacing a subtree of $\Tfr$ of the form of the
left member of~\eqref{equ:congruence_as} into a tree of the form of the
right member of~\eqref{equ:congruence_as}. This transformation is known
in a computer science context as the right rotation
operation~\cite{Knu98} and intervenes in algorithms involving binary
search trees~\cite{AVL62}. The partial order hence generated by the
right rotation operation is known as the Tamari order~\cite{Tam62} and
has a lot of combinatorial and algebraic properties (see for
instance~\cite{HT72,Cha06}).
\medbreak

A first connection between the associative operad and the Tamari order
is based upon the fact that the orientation of~\eqref{equ:congruence_as}
from left to right provides a convergent orientation (a terminating and
confluent rewrite relation) of the congruence $\Congr$. The normal
forms of the rewrite relation induced by the rewrite rule obtained by
orienting~\eqref{equ:congruence_as} from left to right are right comb
binary trees and are hence in one-to-one correspondence with the
elements of~$\As$. Following the ideas developed by Anick for
associative algebras~\cite{Ani86}, the description of an operad by mean
of normal forms provides homological informations for this operad. One of
the fundamental homological properties for operads is
Koszulness~\cite{GK94}, generalizing Koszul associative
algebras~\cite{Pri70}: the convergent orientation
of~\eqref{equ:congruence_as} proves that $\As$ is a Koszul
operad~\cite{DK10,LV12}.
\medbreak

This work is intended to study and collect the possible links between
the Tamari order and some quotients of the operad $\Mag$. In the long
run, the goal is to study quotients $\Mag/_{\Congr}$ of $\Mag$ where
$\Congr$ is an operad congruence generated by equivalence classes of
trees of a fixed degree (that is, a fixed number of internal nodes). In
particular, we would like to know if the fact that $\Congr$ is generated
by equivalence classes of trees forming intervals of the Tamari order
implies algebraic properties for $\Mag/_{\Congr}$ (like the description
of orientations of its space of relations, of nice bases, and of Hilbert
series).
\medbreak

To explore this vast research area, we select to pursue in this paper
the following directions. First, we consider the very general set of the
quotients of $\Mag$ seen as an operad in the category of vector spaces.
We show that this set of operads forms a lattice, wherein its partial
order relation is defined from the existence of operad morphisms
(Theorem~\ref{thm:lattice_structure_of_QMag}). We also provide a
Grassmann formula (see for instance~\cite{Lan02} analog relating the
Hilbert series of the operads of the lattice together with their
lower-bound and upper-bound
(Theorem~\ref{thm:Grassmann_formula_for_Hilbert_series_of_QMag}).
Besides, we focus on a special kind of quotients of $\Mag$, denoted by
$\CAs{\gamma}$, defined by equating the left and right comb binary trees
of a fixed degree $\gamma \geq 1$. Observe that since $\CAs{2}=\As$, the
operads $\CAs{\gamma}$ can be seen as generalizations of $\As$. These
operads are called comb associative operads. For instance, $\CAs{3}$ is
the operad describing the category of the algebras equipped with a
binary product $\Product$ subjected to the relation
\begin{equation}
    \left(\left(x_1 \Product x_2\right) \Product x_3\right) \Product x_4
    =
    x_1 \Product \left(x_2 \Product \left(x_3 \Product x_4\right)\right).
\end{equation}
We first provide general results about the operads $\CAs{\gamma}$. In
particular, we show that the set of these operads forms a lattice which
embeds as a poset in the lattice of the quotients of $\Mag$
aforementioned (Theorems~\ref{thm:lattice_CAs}
and~\ref{thm:inclusion_lattice_CAs}). We focus in particular on the
study of $\CAs{3}$. Observe that the congruence defining this operad is
generated by an equivalence class of trees which is not an interval for
the Tamari order. As preliminary computer experiments show, $\CAs{3}$
has oscillating first dimensions (see~\eqref{equ:dimensions_CAs_3}),
what is rather unusual among all known operads. We provide a convergent
orientation of the space of relations of $\CAs{3}$
(Theorem~\ref{thm:convergent_rewrite_rule_CAs_3}), a description of a
basis of the operad in terms of normal forms, and prove that its Hilbert
series is rational. For all these, we use rewrite systems on
trees~\cite{BN98} and the Buchberger algorithm for operads~\cite{DK10}.
We expose some experimental results obtained with the help of the
computer for some operads $\CAs{\gamma}$ with $\gamma \geq 4$.
All our computer programs are made from scratch in {\sc Caml} and
{\sc Python}. Finally, we continue the investigation of the quotients
of $\Mag$ by regarding the quotients of $\Mag$ obtained by equating two
trees of degree $3$. This leads to ten quotient operads of $\Mag$. We
provide for some of these combinatorial realizations, mostly in terms
of integer compositions.
\medbreak

This text is presented as follows. Section~\ref{sec:operad_Mag} contains
preliminaries about operads, binary trees, the magmatic operad, and
rewrite systems on binary trees. We also prove and recall some important
lemmas about rewrite systems on trees used thereafter. In
Section~\ref{sec:Magmatic_operads}, we study the set of all the
quotients of $\Mag$ seen as an operad in the category of vector spaces
and its lattice structure. Section~\ref{sec:CAs_d} is the heart of this
article and is devoted to the study of the comb associative operads
$\CAs{\gamma}$. Finally, Section~\ref{sec:MAg_3} presents our results
about the quotients of $\Mag$ obtained by equating two trees of
degree~$3$.
\medbreak

Some of the results presented here were announced in~\cite{CCG18}.
\medbreak

\subsubsection*{Acknowledgements}
The authors wish to thank Maxime Lucas for helpful discussions and
Vladimir Dotsenko for his marks of interest and his bibliographic
suggestions.
\medbreak

\subsubsection*{General notations and conventions}
For any integers $a$ and $c$, $[a, c]$ denotes the set
$\{b \in \N : a \leq b \leq c\}$ and $[n]$, the set $[1, n]$. The
cardinality of a finite set $S$ is denoted by~$\# S$. In all this
work, $\K$ is a field of characteristic zero.
\medbreak

\section{The magmatic operad, quotients, and rewrite relations}
\label{sec:operad_Mag}
We set in this preliminary section our notations about operads. We also
provide a definition for the magmatic operad and introduce tools to
handle with some of its quotients involving rewrite systems on binary
trees.
\medbreak

\subsection{Nonsymmetric operads}
A \Def{nonsymmetric operad in the category of sets} (or a
\Def{nonsymmetric operad} for short) is a graded set
\begin{math}
    \Oca = \bigsqcup_{n \geq 1} \Oca(n)
\end{math}
together with maps
\begin{equation}
    \circ_i : \Oca(n) \times \Oca(m) \to \Oca(n + m - 1),
    \qquad 1 \leq i \leq n, 1 \leq m,
\end{equation}
called \Def{partial compositions}, and a distinguished element
$\Unit \in \Oca(1)$, the \Def{unit} of $\Oca$. This data has to satisfy,
for any $x \in \Oca(n)$, $y \in \Oca(m)$, and $z \in \Oca$, the three
relations
\begin{subequations}
\begin{equation} \label{equ:operad_axiom_1}
    (x \circ_i y) \circ_{i + j - 1} z = x \circ_i (y \circ_j z),
    \qquad
    1 \leq i \leq n, 1 \leq j \leq m,
\end{equation}
\begin{equation} \label{equ:operad_axiom_2}
    (x \circ_i y) \circ_{j + m - 1} z = (x \circ_j z) \circ_i y,
    \qquad
    1 \leq i < j \leq n,
\end{equation}
\begin{equation} \label{equ:operad_axiom_3}
    \Unit \circ_1 x = x = x \circ_i \Unit,
    \qquad 1 \leq i \leq n.
\end{equation}
\end{subequations}
Since we consider in this work only nonsymmetric operads, we shall call
these simply \Def{operads}.
\medbreak

Let us provide some elementary definitions and notations about operads.
If $x$ is an element of $\Oca$ such that $x \in \Oca(n)$ for an
$n \geq 1$, the \Def{arity} $|x|$ of $x$ is $n$. The
\Def{complete composition maps} of $\Oca$ are the map
\begin{equation}
    \circ : \Oca(n) \times
    \Oca\left(m_1\right) \times \dots \times \Oca\left(m_n\right)
    \to \Oca\left(m_1 + \dots + m_n\right)
\end{equation}
defined, for any $x \in \Oca(n)$ and $y_1, \dots, y_n \in \Oca$, by
\begin{equation} \label{equ:complete_composition}
    x \circ \left[y_1, \dots, y_n\right] :=
    \left(\cdots \left(\left(x \circ_n y_n\right)
    \circ_{n - 1} y_{n - 1}\right) \cdots\right) \circ_1 y_1.
\end{equation}
If $\Oca_1$ and $\Oca_2$ are two operads, a map
$\phi : \Oca_1 \to \Oca_2$ is an \Def{operad morphism} if it respects
the arities, sends the unit of $\Oca_1$ to the unit of $\Oca_2$, and
commutes with partial composition maps. A map $\phi : \Oca_1 \to \Oca_2$
is an \Def{operad antimorphism} if it respects the arities, sends the
unit of $\Oca_1$ to the unit of $\Oca_2$, and
\begin{math}
    \phi\left(x \circ_i y\right) = \phi(x) \circ_{|x|-i+1} \phi(y)
\end{math}
for all $x, y \in \Oca_1$ and $i \in [|x|]$. We say that $\Oca_2$ is a
\Def{suboperad} of $\Oca_1$ if $\Oca_2$ is a subset of $\Oca_1$
containing the unit of $\Oca_1$ and the partial composition maps of
$\Oca_2$ are the ones of $\Oca_1$ restricted on $\Oca_2$. For any subset
$\GeneratingSet$ of $\Oca$, the \Def{operad generated} by
$\GeneratingSet$ is the smallest suboperad $\Oca^\GeneratingSet$ of
$\Oca$ containing $\GeneratingSet$. When $\Oca = \Oca^\GeneratingSet$,
we say that $\GeneratingSet$ is a \Def{generating set} of $\Oca$. An
\Def{operad congruence} is an equivalence relation $\Congr$ on $\Oca$
such that $\Congr$ respects the arities, and for any
$x, x', y, y' \in \Oca$ such that $x \Congr x'$ and $y \Congr y'$,
$x \circ_i y$ is $\Congr$-equivalent to $x' \circ_i y'$ for any valid
integer $i$. Given an operad congruence $\Congr$, the
\Def{quotient operad} $\Oca/_{\Congr}$ of $\Oca$ by $\Congr$ is the
operad of all $\Congr$-equivalence classes endowed with the partial
composition maps defined in the obvious way. In the case where all the
sets $\Oca(n)$, $n \geq 1$, are finite, the \Def{Hilbert series}
$\HilbertSeries_\Oca(t)$ of $\Oca$ is the series defined by
\begin{equation} \label{equ:hilbert_series}
    \HilbertSeries_\Oca(t) := \sum_{n \geq 1} \# \Oca(n) \; t^n.
\end{equation}
\medbreak

We have provided here definitions about operads in the category of sets.
Nevertheless, operads can be defined in the category of $\K$-vector
spaces. We call
them \Def{linear operads} and we study a class of such operads in
Section~\ref{sec:Magmatic_operads}. All the above definitions extend for
linear operads, mainly by substituting Cartesian products $\times$ of
sets with tensor products $\otimes$ of spaces, maps with linear maps,
operad congruences with operad ideals, and cardinalities of sets with
space dimensions (for instance in~\eqref{equ:hilbert_series}).
If $\Oca$ is an operad in the category of sets, we denote by
$\K \Angle{\Oca}$ the corresponding linear operad defined on the
linear span of $\Oca$, where the partial composition maps of $\Oca$ are
extended by linearity on $\K \Angle{\Oca}$. Conversely, when $\Oca$ is
a linear operad admitting a basis $B$ such that its unit $\Unit$
belongs to $B$ and all partial composition maps are internal in $B$,
$\Oca = \K \Angle{B}$ and this operad can be studied as a
set-theoretic operad~$B$.
\medbreak

\subsection{Binary trees and the magmatic operad}
A \Def{binary tree} is either the \Def{leaf} $\Leaf$ or a pair
$(\Tfr_1, \Tfr_2)$ of binary trees. We use the standard terminology
about binary trees (such as \Def{root}, \Def{internal node},
\Def{left child}, \Def{right child}, {\em etc.}) in this work. Let us
recall the main notions. The \Def{arity} $|\Tfr|$ (resp. \Def{degree}
$\Deg(\Tfr)$) of a binary tree $\Tfr$ is its number of leaves (resp.
internal nodes). A binary tree $\Tfr$ is \Def{quadratic} (resp.
\Def{cubic}) if $\Deg(\Tfr) = 2$ (resp. $\Deg(\Tfr) = 3$). We shall draw
binary trees the root to the top. For instance,
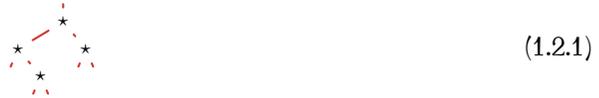
\begin{equation}
    \begin{tikzpicture}[xscale=.15,yscale=.16,Centering]
        \node(0)at(0.00,-4.50){};
        \node(2)at(2.00,-6.75){};
        \node(4)at(4.00,-6.75){};
        \node(6)at(6.00,-4.50){};
        \node(8)at(8.00,-4.50){};
        \node[NodeST](1)at(1.00,-2.25){\begin{math}\Product\end{math}};
        \node[NodeST](3)at(3.00,-4.50){\begin{math}\Product\end{math}};
        \node[NodeST](5)at(5.00,0.00){\begin{math}\Product\end{math}};
        \node[NodeST](7)at(7.00,-2.25){\begin{math}\Product\end{math}};
        \draw[Edge](0)--(1);
        \draw[Edge](1)--(5);
        \draw[Edge](2)--(3);
        \draw[Edge](3)--(1);
        \draw[Edge](4)--(3);
        \draw[Edge](6)--(7);
        \draw[Edge](7)--(5);
        \draw[Edge](8)--(7);
        \node(r)at(5.00,2.25){};
        \draw[Edge](r)--(5);
    \end{tikzpicture}
\end{equation}
is the graphical representation of the binary tree
\begin{math}
    ((\Leaf, (\Leaf, \Leaf)), (\Leaf, \Leaf)).
\end{math}
\medbreak

The \Def{magmatic operad} $\Mag$ is the graded set of all the binary
trees where $\Mag(n)$, $n \geq 1$, is the set of all the binary trees of
arity $n$. The partial composition maps of $\Mag$ are grafting of trees:
given two binary trees $\Tfr$ and $\Sfr$, $\Tfr \circ_i \Sfr$ is the
binary tree obtained by grafting the root of $\Sfr$ onto the $i$th leaf
(numbered from left to right) of $\Tfr$. For instance,
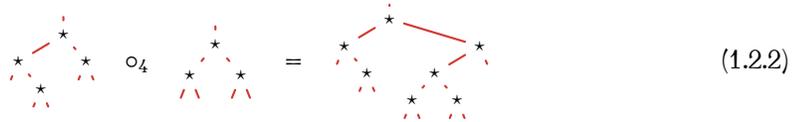
\begin{equation} \label{equ:example_composition_mag}
    \begin{tikzpicture}[xscale=.15,yscale=.16,Centering]
        \node(0)at(0.00,-4.50){};
        \node(2)at(2.00,-6.75){};
        \node(4)at(4.00,-6.75){};
        \node(6)at(6.00,-4.50){};
        \node(8)at(8.00,-4.50){};
        \node[NodeST](1)at(1.00,-2.25){\begin{math}\Product\end{math}};
        \node[NodeST](3)at(3.00,-4.50){\begin{math}\Product\end{math}};
        \node[NodeST](5)at(5.00,0.00){\begin{math}\Product\end{math}};
        \node[NodeST](7)at(7.00,-2.25){\begin{math}\Product\end{math}};
        \draw[Edge](0)--(1);
        \draw[Edge](1)--(5);
        \draw[Edge](2)--(3);
        \draw[Edge](3)--(1);
        \draw[Edge](4)--(3);
        \draw[Edge](6)--(7);
        \draw[Edge](7)--(5);
        \draw[Edge](8)--(7);
        \node(r)at(5.00,2.25){};
        \draw[Edge](r)--(5);
    \end{tikzpicture}
    \enspace \circ_4 \enspace
    \begin{tikzpicture}[xscale=.17,yscale=.18,Centering]
        \node(0)at(0.00,-4.67){};
        \node(2)at(2.00,-4.67){};
        \node(4)at(4.00,-4.67){};
        \node(6)at(6.00,-4.67){};
        \node[NodeST](1)at(1.00,-2.33){\begin{math}\Product\end{math}};
        \node[NodeST](3)at(3.00,0.00){\begin{math}\Product\end{math}};
        \node[NodeST](5)at(5.00,-2.33){\begin{math}\Product\end{math}};
        \draw[Edge](0)--(1);
        \draw[Edge](1)--(3);
        \draw[Edge](2)--(1);
        \draw[Edge](4)--(5);
        \draw[Edge](5)--(3);
        \draw[Edge](6)--(5);
        \node(r)at(3.00,2){};
        \draw[Edge](r)--(3);
    \end{tikzpicture}
    \enspace = \enspace
    \begin{tikzpicture}[xscale=.15,yscale=.12,Centering]
        \node(0)at(0.00,-6.00){};
        \node(10)at(10.00,-12.00){};
        \node(12)at(12.00,-12.00){};
        \node(14)at(14.00,-6.00){};
        \node(2)at(2.00,-9.00){};
        \node(4)at(4.00,-9.00){};
        \node(6)at(6.00,-12.00){};
        \node(8)at(8.00,-12.00){};
        \node[NodeST](1)at(1.00,-3.00){\begin{math}\Product\end{math}};
        \node[NodeST](11)at(11.00,-9.00){\begin{math}\Product\end{math}};
        \node[NodeST](13)at(13.00,-3.00){\begin{math}\Product\end{math}};
        \node[NodeST](3)at(3.00,-6.00){\begin{math}\Product\end{math}};
        \node[NodeST](5)at(5.00,0.00){\begin{math}\Product\end{math}};
        \node[NodeST](7)at(7.00,-9.00){\begin{math}\Product\end{math}};
        \node[NodeST](9)at(9.00,-6.00){\begin{math}\Product\end{math}};
        \draw[Edge](0)--(1);
        \draw[Edge](1)--(5);
        \draw[Edge](10)--(11);
        \draw[Edge](11)--(9);
        \draw[Edge](12)--(11);
        \draw[Edge](13)--(5);
        \draw[Edge](14)--(13);
        \draw[Edge](2)--(3);
        \draw[Edge](3)--(1);
        \draw[Edge](4)--(3);
        \draw[Edge](6)--(7);
        \draw[Edge](7)--(9);
        \draw[Edge](8)--(7);
        \draw[Edge](9)--(13);
        \node(r)at(5.00,2.5){};
        \draw[Edge](r)--(5);
    \end{tikzpicture}
\end{equation}
is a partial composition in $\Mag$. This leads, by definition, to the
following complete composition maps of $\Mag$. Given $\Tfr \in \Mag(n)$
and $\Sfr_1, \dots, \Sfr_n \in \Mag$,
\begin{math}
    \Tfr \circ \left[\Sfr_1, \dots, \Sfr_n\right]
\end{math}
is the binary tree obtained by grafting simultaneously the roots of all
the $\Sfr_i$ onto the $i$th leaves of $\Tfr$. The unit of $\Mag$ is the
leaf. The number of binary trees of arity $n \geq 1$ is the $(n - 1)$st
Catalan number $\Catalan(n - 1)$ and hence, the Hilbert series of
$\Mag$ is
\begin{equation}
    \HilbertSeries_{\Mag}(t)
    = \sum_{n \geq 1} \Catalan(n - 1) t^n
    = \sum_{n \geq 1} \binom{2n - 1}{n - 1} \frac{1}{n} t^n.
\end{equation}
\medbreak

The operad $\Mag$ can be seen as the free operad generated by one binary
element~$\Product$. It satisfies the following universality property.
Let
\begin{math}
    \GeneratingSet := \GeneratingSet(2) := \left\{\Product'\right\}
\end{math}
be the graded set containing exactly one element $\Product'$ of arity
$2$. For any operad $\Oca$ and any map
\begin{math}
    f : \GeneratingSet(2) \to \Oca(2),
\end{math}
there exists a unique operad morphism $\phi : \Mag \to \Oca$ such that
$f = \phi \circ \Corolla$, where $\Corolla$ is the map sending
$\Product'$ to the unique binary tree of degree $1$ (and then, arity
$2$). In other terms, the diagram
\begin{equation}
    \begin{tikzpicture}[xscale=1.3,yscale=0.9,Centering]
        \node(G)at(0,0){\begin{math}\GeneratingSet\end{math}};
        \node(O)at(2,0){\begin{math}\Oca\end{math}};
        \node(AG)at(0,-2){\begin{math}\Mag\end{math}};
        \draw[Map](G)--(O)node[midway,above]{\begin{math}f\end{math}};
        \draw[Injection](G)--(AG)node[midway,left]
            {\begin{math}\Corolla\end{math}};
        \draw[Map,dashed](AG)--(O)node[midway,right]
            {\begin{math}\phi\end{math}};
    \end{tikzpicture}
\end{equation}
commutes.
\medbreak

We now provide some useful tools about binary trees. Given a binary tree
$\Tfr$, we denote by $ \PrefixWord(\Tfr)$ the \Def{prefix word} of
$\Tfr$, that is the word on $\{\Zero, \Two\}$ obtained by a left to
right depth-first traversal of $\Tfr$ and by writing $\Zero$ (resp.
$\Two$) when a leaf (resp. an internal node) is encountered. For
instance,
\begin{equation}
    \begin{tikzpicture}[xscale=.15,yscale=.12,Centering]
        \node(0)at(0.00,-6.00){};
        \node(10)at(10.00,-12.00){};
        \node(12)at(12.00,-12.00){};
        \node(14)at(14.00,-6.00){};
        \node(2)at(2.00,-9.00){};
        \node(4)at(4.00,-9.00){};
        \node(6)at(6.00,-12.00){};
        \node(8)at(8.00,-12.00){};
        \node[NodeST](1)at(1.00,-3.00){\begin{math}\Product\end{math}};
        \node[NodeST](11)at(11.00,-9.00){\begin{math}\Product\end{math}};
        \node[NodeST](13)at(13.00,-3.00){\begin{math}\Product\end{math}};
        \node[NodeST](3)at(3.00,-6.00){\begin{math}\Product\end{math}};
        \node[NodeST](5)at(5.00,0.00){\begin{math}\Product\end{math}};
        \node[NodeST](7)at(7.00,-9.00){\begin{math}\Product\end{math}};
        \node[NodeST](9)at(9.00,-6.00){\begin{math}\Product\end{math}};
        \draw[Edge](0)--(1);
        \draw[Edge](1)--(5);
        \draw[Edge](10)--(11);
        \draw[Edge](11)--(9);
        \draw[Edge](12)--(11);
        \draw[Edge](13)--(5);
        \draw[Edge](14)--(13);
        \draw[Edge](2)--(3);
        \draw[Edge](3)--(1);
        \draw[Edge](4)--(3);
        \draw[Edge](6)--(7);
        \draw[Edge](7)--(9);
        \draw[Edge](8)--(7);
        \draw[Edge](9)--(13);
        \node(r)at(5.00,2.5){};
        \draw[Edge](r)--(5);
    \end{tikzpicture}
    \enspace \xmapsto{\; \PrefixWord \;} \enspace
    \Two \Two \Zero \Two \Zero \Zero \Two \Two \Two \Zero \Zero \Two
    \Zero \Zero \Zero.
\end{equation}
The set of all the words on $\{\Zero, \Two\}$ is endowed with the
lexicographic order $\leq$ induced by $\Zero < \Two$. By extension, this
defines a total order on each set $\Mag(n)$, $n \geq 1$. Indeed, we set
$\Tfr \leq \Tfr'$ if $\Tfr$ and $\Tfr'$ have the same arity and
$\PrefixWord(\Tfr) \leq \PrefixWord(\Tfr')$. Let also the
\Def{left rank} of $\Tfr$ as the number $\LRank(\Tfr)$ of internal
nodes in the left branch beginning at the root of $\Tfr$. For instance,
\begin{equation}
    \begin{tikzpicture}[xscale=.15,yscale=.12,Centering]
        \node(0)at(0.00,-9.00){};
        \node(10)at(10.00,-9.00){};
        \node(12)at(12.00,-9.00){};
        \node(14)at(14.00,-9.00){};
        \node(2)at(2.00,-12.00){};
        \node(4)at(4.00,-12.00){};
        \node(6)at(6.00,-6.00){};
        \node(8)at(8.00,-9.00){};
        \node[NodeST](1)at(1.00,-6.00){\begin{math}\Product\end{math}};
        \node[NodeST](11)at(11.00,-3.00){\begin{math}\Product\end{math}};
        \node[NodeST](13)at(13.00,-6.00){\begin{math}\Product\end{math}};
        \node[NodeST](3)at(3.00,-9.00){\begin{math}\Product\end{math}};
        \node[NodeST](5)at(5.00,-3.00){\begin{math}\Product\end{math}};
        \node[NodeST](7)at(7.00,0.00){\begin{math}\Product\end{math}};
        \node[NodeST](9)at(9.00,-6.00){\begin{math}\Product\end{math}};
        \draw[Edge](0)--(1);
        \draw[Edge](1)--(5);
        \draw[Edge](10)--(9);
        \draw[Edge](11)--(7);
        \draw[Edge](12)--(13);
        \draw[Edge](13)--(11);
        \draw[Edge](14)--(13);
        \draw[Edge](2)--(3);
        \draw[Edge](3)--(1);
        \draw[Edge](4)--(3);
        \draw[Edge](5)--(7);
        \draw[Edge](6)--(5);
        \draw[Edge](8)--(9);
        \draw[Edge](9)--(11);
        \node(r)at(7.00,2.25){};
        \draw[Edge](r)--(7);
    \end{tikzpicture}
    \enspace \xmapsto{\; \LRank\; } \enspace 3.
\end{equation}
Equivalently, $\LRank(\Tfr)$ is the length of the prefix of
$\PrefixWord(\Tfr)$ containing only the letter $\Two$. A binary tree
$\Sfr$ is a \Def{subtree} of $\Tfr$ if it possible to stack $\Sfr$ onto
$\Tfr$ by possibly superimposing leaves of $\Sfr$ onto internal nodes of
$\Tfr$. More formally, by using the operad $\Mag$ and its composition
maps, this is equivalent to the fact that $\Tfr$ expresses as
\begin{equation}
    \Tfr = \Rfr \circ_i
    \left(\Sfr \circ \left[\Rfr_1, \dots, \Rfr_n\right]\right)
\end{equation}
where $\Rfr$ and $\Rfr_1$, \dots, $\Rfr_n$ are binary trees,
$i \in [|\Rfr|]$, and $n$ is the arity of $\Sfr$. When, on the contrary,
$\Sfr$ is not a subtree of $\Tfr$, we say that $\Tfr$ \Def{avoids}
$\Sfr$.
\medbreak

\subsection{Rewrite systems on binary trees}
We present here notions about rewrite systems on binary trees. General
notations and notions appear in~\cite{BN98}.
\medbreak

A \Def{rewrite rule} is an ordered pair $(\Sfr, \Sfr')$ of binary trees
such that $|\Sfr| = |\Sfr'|$. A set $S$ of rewrite rules is a binary
relation on $\Mag$ and it shall be denoted by $\Rew$. We denote by
$\Sfr \Rew \Sfr'$ the fact that $(\Sfr, \Sfr') \in \Rew$. In the sequel,
to define a set of rewrite rules $\Rew$, we shall simply list all the
pairs $\Sfr \Rew \Sfr'$ contained in $\Rew$. The \Def{degree}
$\Deg(\Rew)$ of $\Rew$ is the maximal degree of the binary trees in
relation through $\Rew$. Note that $\Deg(\Rew)$ can be not defined when
$\Rew$ is infinite.
\medbreak

If $\Rew$ is a set of rewrite rules, we denote by $\RewContext$ the
\Def{rewrite relation induced} by $\Rew$. Formally we have
\begin{equation} \label{equ:rewrite_relation_induced}
    \Tfr \circ_i
    \left(\Sfr \circ \left[\Rfr_1, \dots, \Rfr_{n}\right]\right)
    \RewContext
    \Tfr \circ_i
    \left(\Sfr' \circ \left[\Rfr_1, \dots, \Rfr_{n}\right]\right),
\end{equation}
if $\Sfr \Rew \Sfr'$ where $n = |\Sfr|$, and $\Tfr$, $\Rfr_1$, \dots,
$\Rfr_n$ are binary trees. In other words, one has
$\Tfr \RewContext \Tfr'$ if it is possible to obtain $\Tfr'$ from $\Tfr$
by replacing a subtree $\Sfr$ of $\Tfr$ by $\Sfr'$ whenever
$\Sfr \Rew \Sfr'$. For instance, if $\Rew$ is the set of rewrite rules
containing the single rewrite rule
\begin{equation} \label{equ:example_rewrite_rule}
    \begin{tikzpicture}[xscale=.22,yscale=.21,Centering]
        \node(0)at(0.00,-3.50){};
        \node(2)at(2.00,-5.25){};
        \node(4)at(4.00,-5.25){};
        \node(6)at(6.00,-1.75){};
        \node[NodeST](1)at(1.00,-1.75){\begin{math}\Product\end{math}};
        \node[NodeST](3)at(3.00,-3.50){\begin{math}\Product\end{math}};
        \node[NodeST](5)at(5.00,0.00){\begin{math}\Product\end{math}};
        \draw[Edge](0)--(1);
        \draw[EdgeColorF](1)--(5);
        \draw[Edge](2)--(3);
        \draw[EdgeColorF](3)--(1);
        \draw[Edge](4)--(3);
        \draw[Edge](6)--(5);
        \node(r)at(5.00,1.31){};
        \draw[Edge](r)--(5);
    \end{tikzpicture}
    \enspace \Rew \enspace
    \begin{tikzpicture}[xscale=.22,yscale=.21,Centering]
        \node(0)at(0.00,-1.75){};
        \node(2)at(2.00,-3.50){};
        \node(4)at(4.00,-5.25){};
        \node(6)at(6.00,-5.25){};
        \node[NodeST](1)at(1.00,0.00){\begin{math}\Product\end{math}};
        \node[NodeST](3)at(3.00,-1.75){\begin{math}\Product\end{math}};
        \node[NodeST](5)at(5.00,-3.50){\begin{math}\Product\end{math}};
        \draw[Edge](0)--(1);
        \draw[Edge](2)--(3);
        \draw[EdgeColorF](3)--(1);
        \draw[Edge](4)--(5);
        \draw[EdgeColorF](5)--(3);
        \draw[Edge](6)--(5);
        \node(r)at(1.00,1.31){};
        \draw[Edge](r)--(1);
    \end{tikzpicture}\,,
\end{equation}
one has
\begin{equation} \label{equ:example_rewrite_step}
    \begin{tikzpicture}[xscale=.22,yscale=.12,Centering]
        \node(0)at(0.00,-9.00){};
        \node(10)at(10.00,-12.00){};
        \node(12)at(12.00,-18.00){};
        \node(14)at(14.00,-18.00){};
        \node(16)at(16.00,-15.00){};
        \node(18)at(18.00,-9.00){};
        \node(2)at(2.00,-9.00){};
        \node(20)at(20.00,-9.00){};
        \node(4)at(4.00,-6.00){};
        \node(6)at(6.00,-12.00){};
        \node(8)at(8.00,-12.00){};
        \node[NodeST](1)at(1.00,-6.00){\begin{math}\Product\end{math}};
        \node[NodeST](11)at(11.00,-9.00)
            {\begin{math}\Product\end{math}};
        \node[NodeST](13)at(13.00,-15.00)
            {\begin{math}\Product\end{math}};
        \node[NodeST](15)at(15.00,-12.00)
            {\begin{math}\Product\end{math}};
        \node[NodeST](17)at(17.00,-3.00)
            {\begin{math}\Product\end{math}};
        \node[NodeST](19)at(19.00,-6.00)
            {\begin{math}\Product\end{math}};
        \node[NodeST](3)at(3.00,-3.00){\begin{math}\Product\end{math}};
        \node[NodeST](5)at(5.00,0.00){\begin{math}\Product\end{math}};
        \node[NodeST](7)at(7.00,-9.00){\begin{math}\Product\end{math}};
        \node[NodeST](9)at(9.00,-6.00){\begin{math}\Product\end{math}};
        \draw[Edge](0)--(1);
        \draw[Edge](1)--(3);
        \draw[Edge](10)--(11);
        \draw[EdgeColorF](11)--(9);
        \draw[Edge](12)--(13);
        \draw[Edge](13)--(15);
        \draw[Edge](14)--(13);
        \draw[Edge](15)--(11);
        \draw[Edge](16)--(15);
        \draw[Edge](17)--(5);
        \draw[Edge](18)--(19);
        \draw[Edge](19)--(17);
        \draw[Edge](2)--(1);
        \draw[Edge](20)--(19);
        \draw[Edge](3)--(5);
        \draw[Edge](4)--(3);
        \draw[Edge](6)--(7);
        \draw[Edge](7)--(9);
        \draw[Edge](8)--(7);
        \draw[EdgeColorF](9)--(17);
        \node(r)at(5.00,2.25){};
        \draw[Edge](r)--(5);
    \end{tikzpicture}
    \enspace \RewContext \enspace
    \begin{tikzpicture}[xscale=.22,yscale=.12,Centering]
        \node(0)at(0.00,-9.00){};
        \node(10)at(10.00,-9.00){};
        \node(12)at(12.00,-18.00){};
        \node(14)at(14.00,-18.00){};
        \node(16)at(16.00,-15.00){};
        \node(18)at(18.00,-15.00){};
        \node(2)at(2.00,-9.00){};
        \node(20)at(20.00,-15.00){};
        \node(4)at(4.00,-6.00){};
        \node(6)at(6.00,-9.00){};
        \node(8)at(8.00,-9.00){};
        \node[NodeST](1)at(1.00,-6.00){\begin{math}\Product\end{math}};
        \node[NodeST](11)at(11.00,-6.00)
            {\begin{math}\Product\end{math}};
        \node[NodeST](13)at(13.00,-15.00)
            {\begin{math}\Product\end{math}};
        \node[NodeST](15)at(15.00,-12.00)
            {\begin{math}\Product\end{math}};
        \node[NodeST](17)at(17.00,-9.00)
            {\begin{math}\Product\end{math}};
        \node[NodeST](19)at(19.00,-12.00)
            {\begin{math}\Product\end{math}};
        \node[NodeST](3)at(3.00,-3.00){\begin{math}\Product\end{math}};
        \node[NodeST](5)at(5.00,0.00){\begin{math}\Product\end{math}};
        \node[NodeST](7)at(7.00,-6.00){\begin{math}\Product\end{math}};
        \node[NodeST](9)at(9.00,-3.00){\begin{math}\Product\end{math}};
        \draw[Edge](0)--(1);
        \draw[Edge](1)--(3);
        \draw[Edge](10)--(11);
        \draw[EdgeColorF](11)--(9);
        \draw[Edge](12)--(13);
        \draw[Edge](13)--(15);
        \draw[Edge](14)--(13);
        \draw[Edge](15)--(17);
        \draw[Edge](16)--(15);
        \draw[EdgeColorF](17)--(11);
        \draw[Edge](18)--(19);
        \draw[Edge](19)--(17);
        \draw[Edge](2)--(1);
        \draw[Edge](20)--(19);
        \draw[Edge](3)--(5);
        \draw[Edge](4)--(3);
        \draw[Edge](6)--(7);
        \draw[Edge](7)--(9);
        \draw[Edge](8)--(7);
        \draw[Edge](9)--(5);
        \node(r)at(5.00,2.25){};
        \draw[Edge](r)--(5);
    \end{tikzpicture}\,.
\end{equation}
The right member of~\eqref{equ:example_rewrite_step} is obtained by
replacing, in the tree of left member
of~\eqref{equ:example_rewrite_step}, a subtree equal to the left
member of~\eqref{equ:example_rewrite_rule} starting at the right child
of its root by the right member of~\eqref{equ:example_rewrite_rule}.
\medbreak

Let $\Rew$ be a set of rewrite rules and $\RewContext$ be the rewrite
relation induced by $\Rew$. Since~$\RewContext$ is in particular a
binary relation on $\Mag$, the classical notations about closures apply
here: we denote by $\RewContextT$ (resp. $\RewContextRT$,
$\RewContextRST$) the transitive (resp. reflexive and transitive,
and reflexive, symmetric, and transitive) closure of~$\RewContext$.
\medbreak

When $\Tfr_0$, $\Tfr_1$, \dots, $\Tfr_k$ are binary trees such that
\begin{equation}
    \Tfr_0 \RewContext \Tfr_1 \RewContext \cdots \RewContext \Tfr_k,
\end{equation}
we say that $\Tfr_0$ is \Def{rewritable} by $\RewContext$ into $\Tfr_k$
in $k$ \Def{steps}. When there is no infinite chain
\begin{equation} \label{equ:infinite_chain}
    \Tfr_0 \RewContext \Tfr_1 \RewContext \Tfr_2 \RewContext \cdots
\end{equation}
we say that $\RewContext$ is \Def{terminating}. To establish the
termination of a rewrite relation, we will use the following criterion.
\medbreak

\begin{Lemma}\label{lem:prefix_word_termination}
    Let $\Rew$ be a set of rewrite rules on $\Mag$. If for any
    $\Tfr, \Tfr' \in \Mag$ such that $\Tfr \Rew \Tfr'$ one has
    $\Tfr > \Tfr'$, then the rewrite relation induced by $\Rew$ is
    terminating.
\end{Lemma}
\begin{proof}
    Observe first that for any binary trees $\Tfr$ and $\Sfr$, the
    prefix word of $\Tfr \circ_i \Sfr$ is obtained by replacing the
    $i$th $\Zero$ of $\PrefixWord(\Tfr)$ by $\PrefixWord(\Sfr)$. For
    this reason, and due to the
    definition~\eqref{equ:complete_composition} of $\circ$, for any
    binary trees $\Sfr$ and $\Rfr_1, \dots, \Rfr_n$ where $n$ is the
    arity of $\Sfr$, the prefix word of
    $\Sfr \circ \left[\Rfr_1, \dots, \Rfr_n\right]$ is obtained by
    replacing from right to left each $\Zero$ of $\PrefixWord(\Sfr)$ by
    the prefix words of each $\Rfr_i$. This, together with the
    definition~\eqref{equ:rewrite_relation_induced} of the rewrite
    relation $\RewContext$ induced by $\Rew$ and the hypothesis of the
    statement of the lemma, implies that if $\Tfr$ and $\Tfr'$ are two
    binary trees such that $\Tfr \RewContext \Tfr'$,
    $\PrefixWord(\Tfr) > \PrefixWord\left(\Tfr'\right)$. This means
    that $\Tfr > \Tfr'$ and leads to the fact that any chain
    \begin{math}
        \Tfr_0 \RewContext \Tfr_1 \RewContext \Tfr_2 \RewContext \cdots
    \end{math}
    is finite since
    \begin{math}
        \Tfr_0 > \Tfr_1 > \Tfr_2 > \cdots
    \end{math}
    and there is a finite number of binary trees of a fixed arity.
    Therefore, $\RewContext$ is terminating.
\end{proof}
\medbreak

A \Def{normal form} for $\RewContext$ is a binary tree $\Tfr$ such
that for all binary trees $\Tfr'$, $\Tfr \RewContextRT \Tfr'$ implies
$\Tfr' = \Tfr$. In other words, a normal form for $\RewContext$ is a
tree which is not rewritable by $\RewContext$. A normal form for
$\RewContext$ of a binary tree $\Tfr$ is a normal form $\overline{\Tfr}$
for $\RewContext$ such that $\Tfr\RewContextRT\overline{\Tfr}$. When no
confusion is possible, we simply say normal form instead of normal form
for $\RewContext$. The set of all the normal forms is denoted by
$\NormalForms_{\RewContext}$. The trees of $\NormalForms_{\RewContext}$
admit the following description, useful for enumerative prospects.
\medbreak

\begin{Lemma} \label{lem:normal_forms_avoiding}
    Let $\Rew$ be a set of rewrite rules on $\Mag$ and $\RewContext$ be
    the rewrite relation induced by $\Rew$. Then,
    $\NormalForms_{\RewContext}$ is the set of all the binary trees that
    avoid all the trees appearing as left members of~$\Rew$.
\end{Lemma}
\begin{proof}
    Assume first that $\Tfr$ is a binary tree avoiding all the trees
    appearing as left members of~$\Rew$. Then, due to the
    definition~\eqref{equ:rewrite_relation_induced} of $\RewContext$,
    $\Tfr$ is not rewritable by $\RewContext$. Hence, $\Tfr$ is a normal
    form for $\RewContext$. Conversely, assume that
    $\Tfr \in \NormalForms_{\RewContext}$. In this case, by definition
    of a normal form, $\Tfr$ is not rewritable by $\RewContext$, so that
    $\Tfr$ does not admit any occurrence of a tree appearing as a left
    member of~$\Rew$.
\end{proof}
\medbreak

When for all binary trees $\Tfr$, $\Sfr_1$, and $\Sfr_2$ such that
$\Tfr \RewContextRT \Sfr_1$ and $\Tfr \RewContextRT \Sfr_2$, there
exists a binary tree $\Tfr'$ such that $\Sfr_1 \RewContextRT \Tfr'$ and
$\Sfr_2 \RewContextRT \Tfr'$, we say that $\RewContext$ is
\Def{confluent}. Besides, a tree $\Tfr$ is a \Def{branching tree} for
$\RewContext$ if there exists two different trees $\Sfr_1$ and $\Sfr_2$
satisfying $\Tfr \RewContext \Sfr_1$ and $\Tfr \RewContext \Sfr_2$. In
this case, the pair $\{\Sfr_1, \Sfr_2\}$ is a \Def{branching pair} for
$\Tfr$. Moreover, the branching pair $\{\Sfr_1, \Sfr_2\}$ is
\Def{joinable} if there exists a binary tree $\Tfr'$ such that
$\Sfr_1 \RewContextRT \Tfr'$ and $\Sfr_2 \RewContextRT \Tfr'$. The
diamond lemma~\cite{New42} is based upon the inspection of the branching
pairs of a terminating rewrite relation $\RewContext$ in order to prove
its confluence.
\medbreak

\begin{Lemma} \label{lem:diamond_lemma}
    Let $\Rew$ be a set of rewrite rules on $\Mag$ and $\RewContext$ be
    the rewrite relation induced by $\Rew$. Then, if $\RewContext$ is
    terminating and all its branching pairs are joinable, $\RewContext$
    is confluent.
\end{Lemma}
\medbreak

When $\RewContext$ is terminating and confluent, $\RewContext$ is said
\Def{convergent}. We shall use the following result to prove that
a terminating rewrite relation is convergent.
\medbreak

\begin{Lemma} \label{lem:degree_confluence}
    Let $\Rew$ be a set of rewrite rules on $\Mag$ having a degree
    $\Deg(\Rew)$. Then, if the rewrite relation $\RewContext$ induced by
    $\Rew$ is terminating and all its branching pairs made of trees of
    degrees at most $2 \, \Deg(\Rew) - 1$ are joinable, $\RewContext$ is
    convergent.
\end{Lemma}
\begin{proof}
    The statement of the lemma is the specialization on rewrite
    relations on $\Mag$ of a more general result about rewrite relations
    appearing in~\cite[Lemma 1.2.1.]{Gir16}.
\end{proof}
\medbreak

Let us now go back on operads. Let $\Congr$ be an operad congruence of
$\Mag$. If $\Tfr$ is a binary tree, we denote by $[\Tfr]_{\Congr}$ the
$\Congr$-equivalence class of $\Tfr$. By definition, $[\Tfr]_{\Congr}$
is an element of the quotient operad
\begin{equation} \label{equ:quotient_operad_mag}
    \Oca := \Mag/_{\Congr}.
\end{equation}
A set of rewrite rules $\Rew$ on $\Mag$ is an \Def{orientation} of
$\Congr$ if $\RewContextRST$ and $\Congr$ are equal as binary relations,
where $\RewContext$ is the rewrite relation induced by $\Rew$. Moreover,
$\Rew$ is a \Def{convergent} (resp. \Def{terminating}, \Def{confluent})
orientation of $\Congr$ if $\RewContext$ is convergent (resp.
terminating, confluent). In this text, we call \Def{presentation} of a
quotient operad $\Oca$ of the form~\eqref{equ:quotient_operad_mag} the
data of a generating set for the operad congruence $\Congr$. Observe
that any orientation of $\Congr$ is a presentation of $\Oca$, so that
the above nomenclature (\Def{convergent}, \Def{terminating}, and
\Def{confluent}) still holds for presentations. A presentation is said
to be \Def{finite} if it is a finite set.
\medbreak

When $\Rew$ is a convergent orientation of $\Congr$, the set
$\NormalForms_{\RewContext}$ of all normal forms for $\RewContext$ is
called a \Def{Poincaré-Birkhoff-Witt basis}~\cite{Hof10,DK10} (or a
\Def{PBW basis} for short) of the quotient operad $\Oca$. This forms a
one-to-one correspondence between the sets
$\NormalForms_{\RewContext}(n)$ and $\Oca(n)$, $n \geq 1$. In other
words, a PBW basis offers a way to assign with each $\Congr$-equivalence
class $[\Tfr]_{\Congr}$ a representative
\begin{math}
    \Tfr' \in [\Tfr]_{\Congr} \cap \NormalForms_{\RewContext}.
\end{math}
A \Def{combinatorial realization} of an operad $\Oca$ of the
form~\eqref{equ:quotient_operad_mag} is an operad $\Cca$ isomorphic to
$\Oca$ which admits an explicit description of its elements and an
explicit description of its partial composition maps. The knowledge of a
PBW basis $\Cca := \NormalForms_{\RewContext}$ of $\Oca$ provides a
combinatorial realization $\Cca$ of $\Oca$. Indeed, the partial
composition $\Tfr' \circ_i \Sfr'$ of two binary trees $\Tfr'$ and
$\Sfr'$ of $\Cca$ is the tree obtained by grafting the root of $\Sfr'$
onto the $i$th leaf of $\Tfr'$ and by rewriting by $\RewContext$ this
tree as much as possible in order to obtain a normal form. This process
is well-defined since, by hypothesis, $\RewContext$ is convergent.
\medbreak

When $\Rew$ is a terminating but not convergent orientation of $\Congr$,
we shall use a variant of the \Def{Buchberger semi-algorithm} for
operads~\cite[Section 3.7]{DK10} to compute a set of rewrite rules
$\Rew'$ such that, as binary relations $\Rew \subseteq \Rew'$, and
$\Rew'$ is a convergent orientation of $\Congr$. This semi-algorithm
takes as input a finite set of rewrite rules $\Rew$ and outputs the set
of rewrite rules $\Rew'$ satisfying the property stated above. Here is,
step by step, a description of its execution:
\begin{enumerate}[label={(\it\arabic*)}]
    \item Set $\Rew' := \Rew$ and let $\BranchingTrees$ be the set of
    branching trees for $\RewContext$.
    \smallbreak

    \item \label{item:loop_start_buchberger}
    If $\BranchingTrees$ is empty, the execution stops and the output
    is $\Rew'$.
    \smallbreak

    \item\label{item:choice_branching_tree} Otherwise, let $\Tfr$ be a
    branching tree for $\RewContext'$. Remove $\Tfr$ from
    $\BranchingTrees$.
    \smallbreak

    \item\label{item:choice_branching_pair} Let $\{\Sfr_1, \Sfr_2\}$ be
    a branching pair for $\Tfr$.
    \smallbreak

    \item\label{item:computed_normal_forms} Let $\bar{\Sfr_1}$ and
    $\bar{\Sfr_2}$ be normal forms of $\Sfr_1$ and $\Sfr_2$,
    respectively.
    \smallbreak

    \item \label{item:new_rewrite_rule_buchberger}
    If $\bar{\Sfr_1}$ is different from $\bar{\Sfr_2}$, add
    to $\Rew'$ the rewrite rule
    \begin{math}
        \max_\leq \left\{\bar{\Sfr_1}, \bar{\Sfr_2}\right\}
        \Rew'
        \min_\leq \left\{\bar{\Sfr_1}, \bar{\Sfr_2}\right\}.
    \end{math}
    \smallbreak

    \item Add to $\BranchingTrees$ all new branching trees of degrees at
    most $2 \Deg(\Rew') - 1$ created by the rewrite rule created in
    Step~\ref{item:new_rewrite_rule_buchberger}.
    \smallbreak

    \item Go to Step~\ref{item:loop_start_buchberger}.
\end{enumerate}
The set of rewrite rules $\Rew'$ outputted by this semi-algorithm is a
\Def{completion} of $\Rew$. By Lemma~\ref{lem:degree_confluence},
$\RewContext'$ is confluent. Notice that, for certain inputs $\Rew$,
this semi-algorithm never stops. Notice also that the computed
completion depends on the total order $\leq$ on the binary trees of a
same arity, the choices at Steps~\ref{item:choice_branching_tree}
and~\ref{item:choice_branching_pair} as well as the computed normal
forms at Step~\ref{item:computed_normal_forms}.
\medbreak

\section{Quotients of the linear magmatic operad}
\label{sec:Magmatic_operads}
In this section, we equip the set of quotients of the linear magmatic
operad with a lattice structure. We also show a Grassmann formula
analog for this lattice.
\medbreak

\subsection{Lattice structure}
The \Def{linear magmatic operad}, written $\KMag$, is the free linear
operad over one binary generator. By definition, for each arity $n$,
$\KMag(n)$ is the vector space with basis $\Mag(n)$ and the
compositions maps of $\KMag$ are the extensions by linearity the ones
of~$\Mag$.
\medbreak

We denote by $\IMag$ the set of operad ideals of $\KMag$ and we set
\begin{equation} \label{equ:definition_of_QMag}
    \QMag := \Big\{\KMag/_I :  I\in\IMag\Big\},
\end{equation}
as the set of all quotients of the linear magmatic operad. Given
$\KMag/_I \in \QMag$ and $x \in \KMag$, we denote by $[x]_I$ the
$I$-equivalence class of $x$. Observe that $\KMag/_I$ is generated as an
operad by $[\Product]_I$ (where, recall, $\Product$ is the binary
generator of $\Mag$ and thus also of $\KMag$). Moreover, given two
elements $\Oca_1$ and $\Oca_2$ of $\QMag$, we denote by
$\Hom\left(\Oca_1,\Oca_2\right)$ the set of linear operad morphisms from
$\Oca_1$ to~$\Oca_2$.
\medbreak

\begin{Proposition} \label{prop:endomorphisms_of_magmatic_operads}
    For any $\Oca_1, \Oca_2 \in \QMag$, the set
    $\Hom\left(\Oca_1,\Oca_2\right)$ admits a vector space structure.
    Moreover, its dimension is equal to $0$ or $1$ and every nonzero
    morphism is surjective.
\end{Proposition}
\begin{proof}
    Let $I_1, I_2 \in \IMag$ such that $\Oca_1 = \KMag/_{I_1}$ and
    $\Oca_2 = \KMag/_{I_2}$. Since $\Oca_1$ is generated by the binary
    element $[\Product]_{I_1}$, a morphism $\varphi:\Oca_1\to\Oca_2$ is
    uniquely determined by $\varphi\left([\Product]_{I_1}\right)$.
    Moreover, $\varphi\left([\Product]_{I_1}\right)$ has arity $2$ in
    $\Oca_2$. Hence, $\varphi\left([\Product]_{I_1}\right)$ belongs to
    the line spanned by the binary generator of $\Oca_2$, that is there
    exists a scalar $\lambda\in\K$ such that
    \begin{math}
        \varphi\left([\Product]_{I_1}\right)
        = \lambda[\Product]_{I_2}.
    \end{math}
    If there exists such a $\lambda$ different from zero, then for every
    nonzero scalar $\mu$, we have a well-defined operad morphism
    $\psi:\Oca_1\to\Oca_2$ satisfying
    \begin{math}
        \psi([\Product]_{I_1}) =\mu[\Product]_{I_2}
        = \left(\mu \lambda^{-1}\right)\varphi([\Product]_{I_1}).
    \end{math}
    Hence, $\Hom\left(\Oca_1,\Oca_2\right)$ is either reduced to the
    zero morphism or it is in one-to-one correspondence with $\K$, which
    proves that $\Hom\left(\Oca_1,\Oca_2\right)$ is a vector space of
    dimension at most $1$. Moreover, if $\varphi$ is different from $0$,
    that is there is a nonzero scalar such that
    \begin{math}
        \varphi\left([\Product]_{I_1}\right) = \lambda[\Product]_{I_2},
    \end{math}
    we have
    \begin{math}
        \varphi\left(\lambda^{-1}[\Product]_{I_1}\right)
        = [\Product]_{I_2},
    \end{math}
    so that $\varphi$ is surjective.
\end{proof}
\medbreak

We introduce the binary relation $\OrdQMag$ on $\QMag$ as follows: we
have $\Oca_2\OrdQMag\Oca_1$ if the dimension of
$\Hom\left(\Oca_1,\Oca_2\right)$ is equal to~$1$.
\medbreak

\begin{Proposition} \label{prop:order_relations_on_QMag_and_ideals}
    Let $\Oca_1=\KMag/_{I_1}$ and $\Oca_2=\KMag/_{I_2}$ be two operads
    of $\QMag$. We have $\Oca_2\OrdQMag\Oca_1$ if and only if
    $I_1\subseteq I_2$.
\end{Proposition}
\begin{proof}
    Since, by Proposition~\ref{prop:endomorphisms_of_magmatic_operads},
    $\Hom\left(\Oca_1,\Oca_2\right)$ is a vector space of
    dimension at most~$1$, it contains a nonzero morphism if and only if
    the morphism $\bar{\varphi}:\Oca_1\to\Oca_2$ satisfying
    \begin{math}
        \bar{\varphi}\left([\Product]_{I_1}\right) = [\Product]_{I_2}
    \end{math}
    is well-defined, which means that $\Oca_2\OrdQMag\Oca_1$ is
    equivalent to this condition. Moreover, by the universal property
    of the quotient, $\bar{\varphi}$ is well-defined if and only if
    $I_1$ is included in the kernel of the morphism
    $\varphi:\KMag\to\Oca_2$ defined by
    \begin{math}
        \varphi(\Product) = [\Product]_{I_2}.
    \end{math}
    This kernel is equal to $I_2$, so that $\bar{\varphi}$ is
    well-defined if and only if $I_1$ is included in $I_2$, which
    concludes the proof.
\end{proof}
\medbreak

Recall that a \Def{lattice} is a tuple
$\left(E,\leq,\wedge,\vee\right)$ where $\leq$ is a partial order
relation such that any two elements $e$ and $e'$ of $E$ admit a
lower-bound $e\wedge e'$ and an upper-bound $e\vee e'$. In particular,
$\left(\IMag,\subseteq,\cap,+\right)$ is a lattice, where $\cap$ and $+$
are the intersection and the sum of operad ideals, respectively.
\medbreak

Given two operads $\Oca_1=\KMag/_{I_1}$ and $\Oca_2=\KMag/_{I_2}$ of
$\QMag$, let us define
\begin{equation} \label{equ:inf_QMag}
    \Oca_1\InfQMag \Oca_2 := \KMag/_{I_1+I_2}
\end{equation}
and
\begin{equation} \label{equ:sup_QMag}
    \Oca_1\SupQMag \Oca_2 := \KMag/_{I_1\cap I_2}.
\end{equation}
Explicitly, for every positive integer $n$,
$\left(\Oca_1\InfQMag\Oca_2\right)(n)$ (resp.
$\left(\Oca_1\SupQMag\Oca_2\right)(n)$) is the quotient vector space
$\KMag(n)/_{I_1(n)+I_2(n)}$ (resp.
$\KMag(n)/_{I_1(n)\cap I_2(n)}$).
\medbreak

\begin{Theorem} \label{thm:lattice_structure_of_QMag}
    The tuple $\LatQMag$ is a lattice.
\end{Theorem}
\begin{proof}
    First, we observe that the map $\Oca:\IMag\to\QMag$ defined by
    $\Oca(I) := \KMag/_I$ is a bijection: it is surjective by definition
    of $\QMag$ and it is injective since
    \begin{math}
        \Oca\left(I_1\right) = \Oca\left(I_2\right)
    \end{math}
    implies that the kernel of the natural projection
    \begin{math}
        \KMag\to\Oca\left(I_1\right) = \Oca\left(I_2\right)
    \end{math}
    is equal to both $I_1$ and $I_2$. Moreover, from
    Proposition~\ref{prop:order_relations_on_QMag_and_ideals},
    $\Oca_2\OrdQMag\Oca_1$ is equivalent to $I_1\subseteq I_2$, so that
    $\OrdQMag$ is a partial order relation on $\QMag$ and $\Oca$ is a
    decreasing bijection. The tuple
    $\left(\IMag,\subseteq,\cap,+\right)$ being a lattice, the
    decreasing bijection $\Oca$ induces lattice operations on $\QMag$,
    precisely $\InfQMag$ and $\SupQMag$ by definition.
\end{proof}
\medbreak

The union of generating sets of two operad ideals $I_1$ and $I_2$
is a generating set of $I_1+I_2$, so that the union of generating
relations for the two operads $\Oca_1$ and $\Oca_2$ of $\QMag$ forms a
generating set for the relations of $\Oca_1\InfQMag\Oca_2$. However, the
authors do not know how to compute a generating set of the intersection
of ideals (it is not the intersection of the generating relations), so
that we do not know any general method to compute generating relations
for~$\Oca_1\SupQMag\Oca_2$.
\medbreak

\subsection{Hilbert series and Grassmann formula}
The statement of the Grassmann formula analog for $\LatQMag$ is the
following.
\medbreak

\begin{Theorem}
    \label{thm:Grassmann_formula_for_Hilbert_series_of_QMag}
    Let $\Oca_1$ and $\Oca_2$ be two operads of $\QMag$. We have
    \begin{equation} \label{equ:lattice_structure_and_Hiblert_series}
        \HilbertSeries_{\Oca_1\InfQMag\Oca_2}(t)
        +\HilbertSeries_{\Oca_1\SupQMag\Oca_2}(t)
        =\HilbertSeries_{\Oca_1}(t)+\HilbertSeries_{\Oca_2}(t).
    \end{equation}
\end{Theorem}
\begin{proof}
  Let $I_1, I_2 \in \IMag$ be such that $\Oca_1 = \KMag/_{I_1}$ and
    $\Oca_2 = \KMag/_{I_2}$. For every positive integer $n$, we have
  \begin{equation}
    \label{equ:lattice_structure_and_coefficients_of_the_Hiblert_series}
        \dim{\left(\left(\Oca_1\InfQMag\Oca_2\right)(n)\right)}+
        \dim{\left(\left(\Oca_1\SupQMag\Oca_2\right)(n)\right)}=
        \dim{\left(\Oca_1(n)\right)}+\dim{\left(\Oca_2(n)\right)}.
  \end{equation}
  Indeed,
    \begin{equation}\begin{split}
        \dim & {\left(\Oca_1\InfQMag\Oca_2(n)\right)}+
            \dim{\left(\Oca_1\SupQMag\Oca_2(n)\right)}\\
        &=\dim{\Big(\KMag(n)/_{\left(I_1+I_2\right)(n)}\Big)}+
            \dim{\Big(\KMag(n)/_{\left(I_1\cap I_2\right)(n)}\Big)}\\
        &=\dim{\left(\KMag(n)\right)}-\dim{\left(I_1(n)+I_2(n)\right)}
            +\dim{\left(\KMag(n)\right)}
            -\dim{\left(I_1(n)\cap I_2(n)\right)} \\
        &=\dim{\left(\KMag(n)\right)}-\dim{\left(I_1(n)\right)}
            + \dim{\left(\KMag(n)\right)}-\dim{\left(I_2(n)\right)}\\
        &=\dim{\Big(\KMag(n)/_{I_1(n)}\Big)}
            +\dim{\Big(\KMag(n)/_{I_2(n)}\Big)}\\
        &=\dim{\left(\Oca_1(n)\right)}+\dim{\left(\Oca_2(n)\right)}.
    \end{split}\end{equation}
    The third equality is due to the Grassmann
    formula~\cite{Lan02} applied to the subspaces $I_1(n)$ and
    $I_2(n)$ of $\KMag(n)$.
    \smallbreak

    From
    \eqref{equ:lattice_structure_and_coefficients_of_the_Hiblert_series},
    and for every positive integer $n$, the terms of degree $n$ in the
    left and right members
    of~\eqref{thm:Grassmann_formula_for_Hilbert_series_of_QMag} are
    equal, which proves
    Theorem~\ref{thm:Grassmann_formula_for_Hilbert_series_of_QMag}.
\end{proof}
\medbreak

We terminate this section with an example illustrating the lattice
constructions on $\QMag$. For that, we introduce various
operads, which requires the following notations, also used in
Section~\ref{sec:CAs_d}. Let, for any integer $\gamma \geq 1$, the
binary trees $\LComb{\gamma}$ and $\RComb{\gamma}$ be respectively the
left and the right combs of degree $\gamma$. These trees are depicted as
\begin{equation}
  \label{equ:left_and_right_combs}
    \LComb{\gamma} = \enspace
    \begin{tikzpicture}[xscale=.26,yscale=.3,Centering]
        \node(0)at(0.00,-5.25){};
        \node(2)at(2.00,-5.25){};
        \node(4)at(4.00,-3.50){};
        \node(6)at(6.00,-1.75){};
        \node[NodeST](1)at(1.00,-3.50){\begin{math}\Product\end{math}};
        \node[NodeST](3)at(3.00,-1.75){\begin{math}\Product\end{math}};
        \node[NodeST](5)at(5.00,0.00){\begin{math}\Product\end{math}};
        \draw[Edge](0)--(1);
        \draw[Edge,dotted](1)edge[]node[font=\tiny]{
            \begin{math}\gamma\! -\! 1\end{math}\hspace*{.6cm}}(3);
        \draw[Edge](2)--(1);
        \draw[Edge](3)--(5);
        \draw[Edge](4)--(3);
        \draw[Edge](6)--(5);
        \node(r)at(5.00,1.5){};
        \draw[Edge](r)--(5);
    \end{tikzpicture}
    \qquad \mbox{and} \qquad
    \RComb{\gamma} =
    \begin{tikzpicture}[xscale=.26,yscale=.3,Centering]
        \node(0)at(0.00,-1.75){};
        \node(2)at(2.00,-3.50){};
        \node(4)at(4.00,-5.25){};
        \node(6)at(6.00,-5.25){};
        \node[NodeST](1)at(1.00,0.00){\begin{math}\Product\end{math}};
        \node[NodeST](3)at(3.00,-1.75){\begin{math}\Product\end{math}};
        \node[NodeST](5)at(5.00,-3.50){\begin{math}\Product\end{math}};
        \draw[Edge](0)--(1);
        \draw[Edge](2)--(3);
        \draw[Edge](3)--(1);
        \draw[Edge](4)--(5);
        \draw[Edge,dotted](5)edge[]node[font=\tiny]{
            \hspace*{.6cm}\begin{math}\gamma\! -\! 1\end{math}}(3);
        \draw[Edge](6)--(5);
        \node(r)at(1.00,1.5){};
        \draw[Edge](r)--(1);
    \end{tikzpicture}\,,
\end{equation}
where the values on the dotted edges denote the number of internal nodes
they contain.
\medbreak

We first recall that the \Def{linear associative operad} is
\begin{equation} \label{equ:linear_associative_operad}
    \KAs:=\KMag/_{\IAs},
\end{equation}
where $\IAs$ is the ideal spanned
\begin{math}
    \LComb{2}-\RComb{2},
\end{math}
and that its Hilbert series is
\begin{equation} \label{equ:Hilbert_series_of_As}
    \HilbertSeries_{\KAs}(t)=\sum_{n\geq 1}t^n,
\end{equation}
We define the \Def{anti-associative operad} by
\begin{equation}
    \label{equ:anti-associative_operad}
    \AAs:=\KMag/_{\IAAs},
\end{equation}
where $\IAAs$ is the ideal spanned by
\begin{math}
    \LComb{2}+\RComb{2}
\end{math}.
Using the Buchberger algorithm for operads~\cite[Section 3.7]{DK10}, we
check that the set of rewrite rules
\begin{equation}
    \label{equ:convergent_presentation_of_AAS}
    \left\{\LComb{2}\Rew -\RComb{2}, \; \RComb{3}\Rew 0 \right\}
\end{equation}
is a convergent presentation of $\AAs$. We point out that this statement
is false if the characteristic of $\K$ is equal to $2$. Moreover, using
the convergent presentation~\eqref{equ:convergent_presentation_of_AAS},
we have
\begin{equation} \label{equ:Hilbert_series_of_AAs}
    \HilbertSeries_{\AAs}(t)=t+t^2+t^3.
\end{equation}
Let us consider the \Def{2-nilpotent operad}~\cite{Zin12} defined by
\begin{equation} \label{equ:2-nilpotent_operad}
    \TwoNil:=\KMag/_{\INiHil},
\end{equation}
where $\INiHil$ is the ideal spanned by the two trees $\LComb{2}$ and
$\RComb{2}$. We have
\begin{equation} \label{equ:Hilbert_series_of_NiHil}
    \HilbertSeries_{\TwoNil}(t)=t+t^2.
\end{equation}
We introduce, for every integer $\gamma\geq 2$, the (nonlinear)
\Def{$\gamma$-right comb operad} $\RC{\gamma}$ as follows. For every
arity $n$, we let
\begin{equation} \label{equ:compositions_of_RC}
    \RC{\gamma}(n):=
    \begin{cases}
        \Mag(n) &
        \mbox{if } n \leq \gamma, \\
        \RComb{n-1} & \mbox{otherwise},
    \end{cases}
\end{equation}
and the partial composition $\Tfr_1 \circ_i \Tfr_2$ is the partial
composition of $\Tfr_1$ and $\Tfr_2$ in $\Mag$ if the integer
\begin{math}
    n:=|\Tfr_1|+|\Tfr_2|-1
\end{math}
is smaller than or equal to $\gamma$, and $\RComb{n}$ otherwise.
Moreover, by definition of the $\gamma$-right comb operad, we have
\begin{equation} \label{equ:Hilbert_series_of_RC}
    \HilbertSeries_{\RC{\gamma}}(t)=
    \sum_{1\leq n\leq\gamma}\Catalan(n - 1) t^n +\sum_{n\geq\gamma+1}t^n.
\end{equation}
\medbreak

\begin{Lemma} \label{lem:presentation_of_RC}
    We have an isomorphism
    \begin{equation} \label{equ:presentation_of_RC}
        \RC{\gamma}\simeq\Mag/_{\CongrRC{\gamma}},
    \end{equation}
    where $\CongrRC{\gamma}$ is the smallest operad congruence satisfying
    $\Tfr\CongrRC{\gamma}\RComb{\gamma}$, where $\Tfr$ runs over all the
    binary trees of arity $\gamma+1$. In other words, $\RC{\gamma}$ is a
    combinatorial realization of
    \begin{math}
        \Mag/_{\CongrRC{\gamma}}.
    \end{math}
\end{Lemma}
\begin{proof}
    Let $\Rew$ be the set of rewrite rules $\Tfr\Rew\RComb{\gamma}$,
    where $\Tfr$ runs over all the binary trees of arity $\gamma+1$
    different from $\RComb{\gamma}$. The unique normal
    form of arity $n\geq\gamma+1$ for the rewrite relation $\RewContext$
    induced by $\Rew$ is $\RComb{n-1}$, so that $\Rew$ is a convergent
    presentation of
    \begin{math}
        \Mag/_{\CongrRC{\gamma}}.
    \end{math}
    Moreover, the normal forms for $\RewContext$ of arity $n\leq\gamma$
    are all the trees of arity $n$ and, by using the convergent
    presentation $\Rew$, the compositions of
    \begin{math}
        \Mag/_{\CongrRC{\gamma}}
    \end{math}
    satisfy~\eqref{equ:compositions_of_RC}. Hence, $\Rew$ is also a
    convergent presentation of $\RC{\gamma}$ which proves the statement
    of the lemma.
\end{proof}
\medbreak

Now, we define the \Def{linear $\gamma$-right comb operad} $\KRC{\gamma}$
as the linear operad spanned by $\RC{\gamma}$. In particular, its Hilbert
series is given in~\eqref{equ:Hilbert_series_of_RC}, and
Lemma~\ref{lem:presentation_of_RC} implies that we have
\begin{math}
  \KRC{\gamma}=\KAs/_{\IRC{\gamma}}
\end{math}
where $\IRC{\gamma}$ is the ideal spanned by the elements
$\Tfr-\RComb{\gamma}$, with $\Tfr$ a binary tree of arity $\gamma+1$.
\medbreak

The lower-bound and the upper-bound of $\KAs$ and $\AAs$ in the lattice
$\LatQMag$ are described by the following.
\medbreak

\begin{Theorem} \label{thm:example_lattice}
    We have
    \begin{equation} \label{equ:example_lower-bound}
        \KAs\InfQMag\AAs=\TwoNil
    \end{equation}
    and
    \begin{equation} \label{equ:example_upper-bound}
        \KAs\SupQMag\AAs=\KRC{3}.
    \end{equation}
\end{Theorem}
\begin{proof}
    The ideal of relations of $\KAs\InfQMag\AAs$ is equal to
    $\IAs+\IAAs$, so that it is spanned by the two elements
    $\LComb{2}-\RComb{2}$ and $\LComb{2}+\RComb{2}$. By linear
    transformations applied to these generators, $\IAs+\IAAs$ is spanned
    by $\LComb{2}$ and $\RComb{2}$, that is, it is equal to $\INiHil$,
    which proves~\eqref{equ:example_lower-bound}.
    \smallbreak

    Let us now denote by
    \begin{math}
        \pi:\KMag\to\KAs\SupQMag\AAs
    \end{math}
    the natural projection. Let $\Tfr$ be a tree of arity $4$ and let us
    define $\alpha_{\Tfr}:=\Tfr-\RComb{3}$. The elements
    $\alpha_{\Tfr}$ belong to $\IAs$ and to $\IAAs$ since both
    $[\Tfr]_{\IAAs}$ and $[\RComb{3}]_{\IAAs}$ are equal to
    $[0]_{\IAAs}$. The last statement is shown using the convergent
    presentation~\eqref{equ:convergent_presentation_of_AAS} of $\AAs$.
    Hence, the ideal generated by the elements $\alpha_{\Tfr}$, that is
    the ideal of relations of $\KRC{3}$, is included in
    $\IAs\cap\IAAs=\ker(\pi)$, so that $\pi$ induces a surjective
    morphism
    \begin{math}
        \bar{\pi}:\KRC{3}\to\KAs\SupQMag\AAs.
    \end{math}
    We conclude by using Hilbert series:
    $\HilbertSeries_{\KAs\SupQMag\AAs}(t)$ is computed by using the
    Grassmann formula analog with
    Formulas~\eqref{equ:Hilbert_series_of_As},
    \eqref{equ:Hilbert_series_of_AAs},
    and~\eqref{equ:Hilbert_series_of_NiHil}, and it turns out to be
    equal to $\HilbertSeries_{\KRC{3}}(t)$ which is given
    in~\eqref{equ:Hilbert_series_of_RC}. Hence, $\bar{\pi}$ is an
    isomorphism, which proves~\eqref{equ:example_upper-bound}.
\end{proof}
\medbreak

\section{Generalizations of the associative operad}
\label{sec:CAs_d}

In this section, we define comb associative operads and we show that the
set of such operads admits a lattice structure, isomorphic to the
lattice of division for nonnegative integers. We relate this lattice to
the one of the linear magmatic quotients considered in the previous
section. We also provide a finite convergent presentation of the comb
associative operad corresponding to~$3$.
\medbreak

\subsection{Comb associative operads}
Recall first that the \Def{associative operad} $\As$ is the quotient
of $\Mag$ by the smallest operad congruence $\Congr$ satisfying
\begin{equation}
    \begin{tikzpicture}[xscale=.24,yscale=.24,Centering]
        \node(0)at(0.00,-3.33){};
        \node(2)at(2.00,-3.33){};
        \node(4)at(4.00,-1.67){};
        \node[NodeST](1)at(1.00,-1.67)
            {\begin{math}\Product\end{math}};
        \node[NodeST](3)at(3.00,0.00)
            {\begin{math}\Product\end{math}};
        \draw[Edge](0)--(1);
        \draw[Edge](1)--(3);
        \draw[Edge](2)--(1);
        \draw[Edge](4)--(3);
        \node(r)at(3.00,1.5){};
        \draw[Edge](r)--(3);
    \end{tikzpicture}
    \Congr
    \begin{tikzpicture}[xscale=.24,yscale=.24,Centering]
        \node(0)at(0.00,-1.67){};
        \node(2)at(2.00,-3.33){};
        \node(4)at(4.00,-3.33){};
        \node[NodeST](1)at(1.00,0.00)
                {\begin{math}\Product\end{math}};
        \node[NodeST](3)at(3.00,-1.67)
                {\begin{math}\Product\end{math}};
        \draw[Edge](0)--(1);
        \draw[Edge](2)--(3);
        \draw[Edge](3)--(1);
        \draw[Edge](4)--(3);
        \node(r)at(1.00,1.5){};
        \draw[Edge](r)--(1);
    \end{tikzpicture}\,.
\end{equation}
We propose here a generalization of $\Congr$ in order to define
generalizations of $\As$.
\medbreak

As in Section~\ref{sec:Magmatic_operads}, the left and the right combs of
degree $\gamma$ are denoted by $\LComb{\gamma}$ and $\RComb{\gamma}$,
respectively. In the sequel , we shall employ the drawing convention
introduced after~\eqref{equ:left_and_right_combs}: the values on dotted
edges in a binary tree denote the number of internal nodes they contain.
Moreover, we also employ the convention stipulating that dotted edges
with no value have any number of internal nodes. Let us now define for
any $\gamma \geq 1$ the \Def{$\gamma$-comb associative operad}
$\CAs{\gamma}$ as the quotient operad $\Mag/_{\CongrCAs{\gamma}}$ where
$\CongrCAs{\gamma}$ is the smallest operad congruence of $\Mag$
satisfying
\begin{equation} \label{equ:congruence_CAs_gamma}
    \LComb{\gamma} \enspace \CongrCAs{\gamma} \enspace \RComb{\gamma}.
\end{equation}
Notice that $\CongrCAs{1}$ is trivial so that
$\CAs{1} = \Mag$, and that $\CongrCAs{2}$ is the operad congruence
defining $\As$ so that $\CAs{2} = \As$. Let also
\begin{equation}
    \CAsAll := \left\{\CAs{\gamma} : \gamma\geq 1\right\}
\end{equation}
be the set of all the $\gamma$-comb associative operads.
\medbreak

\subsection{Lattice of comb associative operads}
In order to introduce a lattice structure on $\CAsAll$, we begin by
studying operad morphisms between its elements by mean of intermediate
lemmas.
\medbreak

\begin{Lemma} \label{lem:first_dimensions_CAs}
    For all positive integers $\gamma$ and $n$ such that $\gamma \geq 2$
    and $n \leq \gamma + 1$,
    \begin{equation}
        \# \CAs{\gamma}(n) =
        \Catalan(n - 1) - \delta_{n, \gamma + 1},
    \end{equation}
    where $\delta_{x, y}$ is the Kronecker delta.
\end{Lemma}
\begin{proof}
    Since the equivalence relation $\CongrCAs{\gamma}$ is trivial on the
    binary trees of degrees $d < \gamma$, and since a binary tree of
    degree $d$ has arity $n := d + 1$, one has
    $\# \CAs{\gamma}(n) = \# \Mag(n) = \Catalan(n - 1)$ with
    $n \leq \gamma$. Besides, by definition of $\CongrCAs{\gamma}$, all
    the $\CongrCAs{\gamma}$-equivalence classes of binary trees of
    degree $\gamma$ are trivial, except one due to the fact that
    $\LComb{\gamma} \ne \RComb{\gamma}$ and
    $\LComb{\gamma} \CongrCAs{\gamma} \RComb{\gamma}$. Therefore, since
    a binary tree of degree $\gamma$ has arity $n := \gamma + 1$,
    \begin{math}
        \# \CAs{\gamma}(n)
        = \# \Mag(\gamma + 1) - 1
        = \Catalan(\gamma + 1 - 1) - 1
        = \Catalan(n - 1) - 1
    \end{math}
    as stated.
\end{proof}
\medbreak

\begin{Lemma} \label{lem:surjective_morphisms_CAs}
    Let $\gamma$ and $\gamma'$ be two positive integers. If there exists
    an operad morphism $\varphi:\CAs{\gamma'} \to \CAs{\gamma}$, then it
    is surjective and satisfies
    \begin{math}
        \varphi\left(\left[\Tfr\right]_{\CongrCAs{\gamma'}}\right)
        = \left[\Tfr\right]_{\CongrCAs{\gamma}}
    \end{math}
    for any binary tree~$\Tfr$.
\end{Lemma}
\begin{proof}
    The operad $\CAs{\gamma'}$ is generated by one binary generator
    $[\Product]_{\CongrCAs{\gamma'}}$, which is the image of the binary
    generator $\Product$ of $\Mag$ in $\CAs{\gamma'}$. Hence, $\varphi$
    is entirely determined by the image
    $\varphi\left([\Product]_{\CongrCAs{\gamma'}}\right)$. Moreover,
    $\varphi([\Product]_{\CongrCAs{\gamma'}})$ has to be of arity $2$ in
    $\CAs{\gamma}$, so that we necessarily have
    \begin{math}
        \varphi\left([\Product]_{\CongrCAs{\gamma'}}\right)
        =
        [\Product]_{\CongrCAs{\gamma}}.
    \end{math}
    Hence, if $\varphi$ exists, it is the unique operad morphism from
    $\CAs{\gamma'}$ to $\CAs{\gamma}$ determined by the image of
    $[\Product]_{\CongrCAs{\gamma'}}$. In this case,
    $[\Product]_{\CongrCAs{\gamma}}$ being in the image of $\varphi$,
    the latter is surjective. Finally, it follows that $\varphi$ sends
    $\left[\Tfr\right]_{\CongrCAs{\gamma'}}$ to
    $\left[\Tfr\right]_{\CongrCAs{\gamma}}$ by induction on the degree
    of the binary tree~$\Tfr$.
\end{proof}
\medbreak

\begin{Lemma} \label{lem:injective_morphisms_CAs}
    Let $\gamma$ and $\gamma'$ be two positive integers and
    $\varphi:\CAs{\gamma'} \to \CAs{\gamma}$ be an operad morphism.
    Then, $\varphi$ is injective if and only if $\gamma = \gamma'$.
\end{Lemma}
\begin{proof}
    Assume that $\varphi$ is injective. By
    Lemma~\ref{lem:surjective_morphisms_CAs}, $\varphi$ is also
    surjective, so that $\varphi$ is an isomorphism. If
    $\gamma \ne \gamma'$, by Lemma~\ref{lem:first_dimensions_CAs}, there
    is a positive integer $n$ such that
    $\# \CAs{\gamma}(n) \ne \# \CAs{\gamma'}(n)$. This is contradictory
    with the fact that $\CAs{\gamma}$ and $\CAs{\gamma'}$ are
    isomorphic. Hence, $\gamma = \gamma'$.
    \smallbreak

    Conversely, if $\gamma = \gamma'$, the only operad morphism from
    $\CAs{\gamma}$ to itself sends the generator
    $[\Product]_{\CongrCAs{\gamma}}$ to itself. This maps extends as
    an operad morphism into the identity morphism which is of course
    injective.
\end{proof}
\medbreak

\begin{Lemma} \label{lem:morphism_CAs}
    Let $\gamma$ and $\gamma'$ be two positive integers. There exists an
    operad morphism $\varphi:\CAs{\gamma'} \to \CAs{\gamma}$ if and only
    if
    \begin{math}
      \LComb{\gamma'} \CongrCAs{\gamma} \RComb{\gamma'}.
    \end{math}
\end{Lemma}
\begin{proof}
    Assume that $\varphi:\CAs{\gamma'} \to \CAs{\gamma}$ is an operad
    morphism. Since
    \begin{math}
      \LComb{\gamma'} \CongrCAs{\gamma'} \RComb{\gamma'},
    \end{math}
    we have
    \begin{equation} \label{equ:morphism_CAs}
        \varphi\left(
        \left[\LComb{\gamma'}\right]_{\CongrCAs{\gamma'}}
        \right)
        =
        \varphi\left(
        \left[\RComb{\gamma'}\right]_{\CongrCAs{\gamma'}}
        \right).
    \end{equation}
    Now, by using Lemma~\ref{lem:surjective_morphisms_CAs}, we obtain
    from~\eqref{equ:morphism_CAs} the relation
    \begin{equation}
        \left[\LComb{\gamma'}\right]_{\CongrCAs{\gamma}}
        =
        \left[\RComb{\gamma'}\right]_{\CongrCAs{\gamma}},
    \end{equation}
    saying that
    \begin{math}
        \LComb{\gamma'} \CongrCAs{\gamma} \RComb{\gamma'}
    \end{math}
    as excepted.
    \smallbreak

    Conversely, when
    \begin{math}
        \LComb{\gamma'}\CongrCAs{\gamma}\RComb{\gamma'},
    \end{math}
    let $\varphi:\CAs{\gamma'}(2) \to \CAs{\gamma}(2)$ be the map
    defined by
    \begin{math}
        \varphi\left(
        \left[\Product\right]_{\CongrCAs{\gamma'}}\right)
        :=
        \left[\Product\right]_{\CongrCAs{\gamma}}.
    \end{math}
    Now, since $\CongrCAs{\gamma}$ is coarser than $\CongrCAs{\gamma'}$,
    $\varphi$ extends (in a unique way) into an operad morphism, whence
    the statement of the lemma.
\end{proof}
\medbreak

We define the binary relation $\OrdCAs$ on $\CAsAll$ as follows: we have
$\CAs{\gamma} \OrdCAs \CAs{\gamma'}$ if and only if there exists a
morphism $\varphi:\CAs{\gamma'} \to \CAs{\gamma}$.
\medbreak

\begin{Proposition}\label{prop:poset_CAs}
    The binary relation $\OrdCAs$ is a partial order relation on
    $\CAsAll$.
\end{Proposition}
\begin{proof}
    The binary relation $\OrdCAs$ is reflexive since there exists the
    identity morphism on $\CAs{\gamma}$ for every positive integer
    $\gamma$. It is transitive since the composite of operad morphisms is
    an operad morphism. Finally, let us assume that there exist  tow
    morphisms $\varphi:\CAs{\gamma'}\to\CAs{\gamma}$ and
    $\psi:\CAs{\gamma}\to\CAs{\gamma'}$. In particular,
    $\psi \circ \varphi$ and $\varphi \circ \psi$ are endomorphisms of
    $\CAs{\gamma'}$ and $\CAs{\gamma}$, respectively. From
    Lemma~\ref{lem:surjective_morphisms_CAs}, these two morphisms are
    identity morphisms, so that $\varphi$ and $\psi$ are injective. From
    Lemma~\ref{lem:injective_morphisms_CAs}, $\gamma$ and $\gamma'$ are
    equal, which proves that $\OrdCAs$ is anti-symmetric. Hence,
    $\OrdCAs$ is a partial order.
\end{proof}
\medbreak

In order to show that $\left(\CAsAll, \OrdCAs\right)$ extends into a
lattice, we relate $\left(\CAsAll, \OrdCAs\right)$ with the lattice of
integers $\left(\N,\mid, \gcd, \lcm\right)$, where $\mid$ denotes the
division relation, $\gcd$ denotes the greatest common divisor, and
$\lcm$ the least common multiple operators, respectively.
\medbreak

Recall that $\LRank(\Tfr)$ denotes the left rank of a binary tree
$\Tfr$, as defined in Section~\ref{sec:operad_Mag}. Besides, to simplify
the notation, we shall write $\bar{a}$ instead of $a - 1$ for any
integer~$a$.
\medbreak

\begin{Lemma}\label{lem:left_rank_and_CongrCAs}
    Let $\gamma \geq 2$ be an integer and let $\Tfr$ and $\Tfr'$ be two
    binary trees. If $\Tfr \CongrCAs{\gamma} \Tfr'$, then
    \begin{equation} \label{equ:left_rank_and_CongrCAs}
        \LRank(\Tfr) \pmod{\bar{\gamma}}
        \enspace = \enspace
        \LRank\left(\Tfr'\right) \pmod{\bar{\gamma}}.
    \end{equation}
\end{Lemma}
\begin{proof}
    Consider here the rewrite rule $\Rew$ on $\Mag$ satisfying
    $\LComb{\gamma} \Rew \RComb{\gamma}$. Let us show that the rewrite
    relation $\RewContext$ induced by $\Rew$ is such that
    $\Tfr \RewContext \Tfr'$ implies~\eqref{equ:left_rank_and_CongrCAs}.
    Any binary tree $\Tfr$ decomposes as
    \begin{math}
        \Tfr =
        \LComb{\LRank\left(\Tfr\right)} \circ
        \left[\Leaf, \Tfr_1, \dots,
        \Tfr_{\LRank\left(\Tfr\right)} \right]
    \end{math}
    where the $\Tfr_i$ are binary trees. Now, if
    $\Tfr \RewContext \Tfr'$, then one among the following two cases
    occurs.
    \begin{enumerate}[label={(\it\roman*)}]
        \item The rewrite step is applied into one of the trees
        $\Tfr_i$, that is there exists $\Tfr_i'$ such that
        \begin{math}
            \Tfr' = \LComb{\LRank\left(\Tfr\right)}
            \circ \left[\Leaf, \Tfr_1, \dots, \Tfr_i', \dots,
                \Tfr_{\LRank\left(\Tfr\right)}
            \right]
        \end{math}
        so that
        \begin{math}
            \LRank\left(\Tfr'\right) = \LRank\left(\Tfr\right).
        \end{math}
        \item The rewrite step is applied into the left branch
          beginning
            at the root of $\Tfr$, that is there exists $i$ such that
            \begin{math}
                \Tfr' = \LComb{\LRank\left(\Tfr\right)-\bar{\gamma}}
                \circ \left[\Leaf, \Tfr_1, \dots, \RComb{\bar{\gamma}}
                \circ
                \left[\Tfr_{i},\dots, \Tfr_{i+\bar{\gamma}}\right],
                \dots, \Tfr_{\LRank\left(\Tfr\right)} \right]
            \end{math}
            so that
            \begin{math}
                \LRank\left(\Tfr'\right) =
                \LRank\left(\Tfr\right) - \bar{\gamma} =
                \LRank\left(\Tfr\right) \pmod{\bar{\gamma}}
            \end{math}.
    \end{enumerate}
    This implies~\eqref{equ:left_rank_and_CongrCAs}. Finally, since
    $\CongrCAs{\gamma}$ is the reflexive, symmetric, and transitive
    closure of~$\RewContext$, the statement of the lemma follows.
\end{proof}
\medbreak

\begin{Proposition} \label{prop:division_CAs}
    Let $\gamma$ and $\gamma'$ be two positive integers such that
    $\gamma \geq 2$. Then, there
    exists a morphism $\varphi:\CAs{\gamma'} \to \CAs{\gamma}$ if and
    only if $\bar{\gamma} \mid \bar{\gamma'}$.
\end{Proposition}
\begin{proof}
    From Lemma~\ref{lem:morphism_CAs}, it is enough to show that
    $\LComb{\gamma'} \CongrCAs{\gamma} \RComb{\gamma'}$ if and only if
    $\bar{\gamma} \mid \bar{\gamma'}$. If
    \begin{math}
        \LComb{\gamma'}\CongrCAs{\gamma}\RComb{\gamma'},
    \end{math}
    as a consequence of the existence of a surjective morphism $\varphi$
    from $\CAs{\gamma'}$ to $\CAs{\gamma}$ and
    Lemmas~\ref{lem:first_dimensions_CAs}
    and~\ref{lem:surjective_morphisms_CAs}, one has $\gamma' = 1$ or
    $\gamma \leq \gamma'$. Since
    \begin{equation}
        \LRank\left(\LComb{\gamma'}\right)
        - \LRank\left(\RComb{\gamma'}\right)
        = \gamma' - 1
        = \bar{\gamma'},
    \end{equation}
    by using Lemma~\ref{lem:left_rank_and_CongrCAs}, we deduce that
    $\bar{\gamma'}$ is divisible by $\bar{\gamma}$, which shows the
    direct implication.
    \smallbreak

    Conversely, if $\bar{\gamma} \mid \bar{\gamma'}$, the rewrite rule
    $\Rew$ on $\Mag$ satisfying $\LComb{\gamma} \Rew \RComb{\gamma}$
    induces the sequence
    \begin{multline}
        \LComb{\gamma'} \enspace = \enspace
        \begin{tikzpicture}[xscale=.19,yscale=.25,Centering]
            \node(0)at(0.00,-13.12){};
            \node(10)at(10.00,-5.62){};
            \node(12)at(12.00,-3.75){};
            \node(14)at(14.00,-1.88){};
            \node(2)at(2.00,-13.12){};
            \node(4)at(4.00,-11.25){};
            \node(6)at(6.00,-9.38){};
            \node(8)at(8.00,-7.50){};
            \node[NodeST](1)at(1.00,-11.25)
                {\begin{math}\Product\end{math}};
            \node[NodeST](11)at(11.00,-1.88)
                {\begin{math}\Product\end{math}};
            \node[NodeST](13)at(13.00,0.00)
                {\begin{math}\Product\end{math}};
            \node[NodeST](3)at(3.00,-9.38)
                {\begin{math}\Product\end{math}};
            \node[NodeST](5)at(5.00,-7.50)
                {\begin{math}\Product\end{math}};
            \node[NodeST](7)at(7.00,-5.62)
                {\begin{math}\Product\end{math}};
            \node[NodeST](9)at(9.00,-3.75)
                {\begin{math}\Product\end{math}};
            \draw[Edge](0)--(1);
            \draw[Edge,dotted](1)edge[]node[font=\footnotesize]{
                \begin{math}\bar{\gamma}\end{math}\hspace*{.5cm}}(3);
            \draw[Edge](10)--(9);
            \draw[Edge](11)--(13);
            \draw[Edge](12)--(11);
            \draw[Edge](14)--(13);
            \draw[Edge](2)--(1);
            \draw[Edge,dotted](3)--(5);
            \draw[Edge](4)--(3);
            \draw[Edge,dotted](5)edge[]node[font=\footnotesize]{
                \begin{math}\bar{\gamma}\end{math}\hspace*{.5cm}}(7);
            \draw[Edge](6)--(5);
            \draw[Edge](7)--(9);
            \draw[Edge](8)--(7);
            \draw[Edge,dotted](9)edge[]node[font=\footnotesize]{
                \begin{math}\bar{\gamma}\end{math}\hspace*{.5cm}}(11);
            \node(r)at(13.00,1.41){};
            \draw[Edge](r)--(13);
        \end{tikzpicture}
        \enspace \RewContext \enspace
        \begin{tikzpicture}[xscale=.19,yscale=.2,Centering]
            \node(0)at(0.00,-12.50){};
            \node(10)at(10.00,-5.00){};
            \node(12)at(12.00,-7.50){};
            \node(14)at(14.00,-7.50){};
            \node(2)at(2.00,-12.50){};
            \node(4)at(4.00,-10.00){};
            \node(6)at(6.00,-7.50){};
            \node(8)at(8.00,-5.00){};
            \node[NodeST](1)at(1.00,-10.00)
                {\begin{math}\Product\end{math}};
            \node[NodeST](11)at(11.00,-2.50)
                {\begin{math}\Product\end{math}};
            \node[NodeST](13)at(13.00,-5.00)
                {\begin{math}\Product\end{math}};
            \node[NodeST](3)at(3.00,-7.50)
                {\begin{math}\Product\end{math}};
            \node[NodeST](5)at(5.00,-5.00)
                {\begin{math}\Product\end{math}};
            \node[NodeST](7)at(7.00,-2.50)
                {\begin{math}\Product\end{math}};
            \node[NodeST](9)at(9.00,0.00)
                {\begin{math}\Product\end{math}};
            \draw[Edge](0)--(1);
            \draw[Edge,dotted](1)edge[]node[font=\footnotesize]{
                \begin{math}\bar{\gamma}\end{math}\hspace*{.5cm}}(3);
            \draw[Edge](10)--(11);
            \draw[Edge](11)--(9);
            \draw[Edge](12)--(13);
            \draw[Edge,dotted](13)edge[]node[font=\footnotesize]{
                \hspace*{.5cm}\begin{math}\bar{\gamma}\end{math}}(11);
            \draw[Edge](14)--(13);
            \draw[Edge](2)--(1);
            \draw[Edge,dotted](3)--(5);
            \draw[Edge](4)--(3);
            \draw[Edge,dotted](5)edge[]node[font=\footnotesize]{
                \begin{math}\bar{\gamma}\end{math}\hspace*{.5cm}}(7);
            \draw[Edge](6)--(5);
            \draw[Edge](7)--(9);
            \draw[Edge](8)--(7);
            \node(r)at(9.00,1.88){};
            \draw[Edge](r)--(9);
        \end{tikzpicture} \\
        \enspace \RewContext \enspace
        \begin{tikzpicture}[xscale=.19,yscale=.2,Centering]
            \node(0)at(0.00,-7.50){};
            \node(10)at(10.00,-10.00){};
            \node(12)at(12.00,-12.50){};
            \node(14)at(14.00,-12.50){};
            \node(2)at(2.00,-7.50){};
            \node(4)at(4.00,-5.00){};
            \node(6)at(6.00,-5.00){};
            \node(8)at(8.00,-7.50){};
            \node[NodeST](1)at(1.00,-5.00)
                {\begin{math}\Product\end{math}};
            \node[NodeST](11)at(11.00,-7.50)
                {\begin{math}\Product\end{math}};
            \node[NodeST](13)at(13.00,-10.00)
                {\begin{math}\Product\end{math}};
            \node[NodeST](3)at(3.00,-2.50)
                {\begin{math}\Product\end{math}};
            \node[NodeST](5)at(5.00,0.00)
                {\begin{math}\Product\end{math}};
            \node[NodeST](7)at(7.00,-2.50)
                {\begin{math}\Product\end{math}};
            \node[NodeST](9)at(9.00,-5.00)
                {\begin{math}\Product\end{math}};
            \draw[Edge](0)--(1);
            \draw[Edge,dotted](1)edge[]node[font=\footnotesize]{
                \begin{math}\bar{\gamma}\end{math}\hspace*{.5cm}}(3);
            \draw[Edge](10)--(11);
            \draw[Edge](11)--(9);
            \draw[Edge](12)--(13);
            \draw[Edge,dotted](13)edge[]node[font=\footnotesize]{
                \hspace*{.5cm}\begin{math}\bar{\gamma}\end{math}}(11);
            \draw[Edge](14)--(13);
            \draw[Edge](2)--(1);
            \draw[Edge,dotted](3)--(5);
            \draw[Edge](4)--(3);
            \draw[Edge](6)--(7);
            \draw[Edge](7)--(5);
            \draw[Edge](8)--(9);
            \draw[Edge,dotted](9)edge[]node[font=\footnotesize]{
                \hspace*{.5cm}\begin{math}\bar{\gamma}\end{math}}(7);
            \node(r)at(5.00,1.88){};
            \draw[Edge](r)--(5);
        \end{tikzpicture}
        \enspace \RewContextRT \enspace
        \begin{tikzpicture}[xscale=.19,yscale=.25,Centering]
            \node(0)at(0.00,-1.88){};
            \node(10)at(10.00,-11.25){};
            \node(12)at(12.00,-13.12){};
            \node(14)at(14.00,-13.12){};
            \node(2)at(2.00,-3.75){};
            \node(4)at(4.00,-5.62){};
            \node(6)at(6.00,-7.50){};
            \node(8)at(8.00,-9.38){};
            \node[NodeST](1)at(1.00,0.00)
                {\begin{math}\Product\end{math}};
            \node[NodeST](11)at(11.00,-9.38)
                {\begin{math}\Product\end{math}};
            \node[NodeST](13)at(13.00,-11.25)
                {\begin{math}\Product\end{math}};
            \node[NodeST](3)at(3.00,-1.88)
                {\begin{math}\Product\end{math}};
            \node[NodeST](5)at(5.00,-3.75)
                {\begin{math}\Product\end{math}};
            \node[NodeST](7)at(7.00,-5.62)
                {\begin{math}\Product\end{math}};
            \node[NodeST](9)at(9.00,-7.50)
                {\begin{math}\Product\end{math}};
            \draw[Edge](0)--(1);
            \draw[Edge](10)--(11);
            \draw[Edge,dotted](11)--(9);
            \draw[Edge](12)--(13);
            \draw[Edge,dotted](13)edge[]node[font=\footnotesize]{
                \hspace*{.5cm}\begin{math}\bar{\gamma}\end{math}}(11);
            \draw[Edge](14)--(13);
            \draw[Edge](2)--(3);
            \draw[Edge](3)--(1);
            \draw[Edge](4)--(5);
            \draw[Edge,dotted](5)edge[]node[font=\footnotesize]{
                \hspace*{.5cm}\begin{math}\bar{\gamma}\end{math}}(3);
            \draw[Edge](6)--(7);
            \draw[Edge](7)--(5);
            \draw[Edge](8)--(9);
            \draw[Edge,dotted](9)edge[]node[font=\footnotesize]{
                \hspace*{.5cm}\begin{math}\bar{\gamma}\end{math}}(7);
            \node(r)at(1.00,1.41){};
            \draw[Edge](r)--(1);
        \end{tikzpicture}
        \enspace = \enspace
        \RComb{\gamma'}
    \end{multline}
    of rewrite steps, where dotted edges denotes left or right
    comb trees of degree $\gamma - 1$. Hence, since
    $\CongrCAs{\gamma}$ is the reflexive, symmetric, and transitive
    closure of $\RewContext$, we
    have~$\LComb{\gamma'}\CongrCAs{\gamma}\RComb{\gamma'}$.
\end{proof}
\medbreak

Proposition~\ref{prop:division_CAs} implies the following result.
\medbreak

\begin{Theorem}\label{thm:lattice_CAs}
  The tuple $\LatCAs$ is a lattice, where $\InfCAs$ and $\SupCAs$ are
  defined, for all positive integers $\gamma$ and $\gamma'$, by
    \begin{equation}
      \CAs{\gamma}\InfCAs\CAs{\gamma'}
      :=\CAs{\gcd \left(\bar{\gamma}, \bar{\gamma'}\right)}
    \end{equation}
    and
    \begin{equation}
      \CAs{\gamma}\SupCAs\CAs{\gamma'}
      :=\CAs{\lcm \left(\bar{\gamma}, \bar{\gamma'}\right)}.
    \end{equation}
\end{Theorem}
\medbreak

The lattice $\left(\N, \mid, \gcd, \lcm\right)$ admits $1$ as minimum
element and $0$ as maximum element since any nonnegative integer is
divisible by $1$ and divides $0$. These properties translate as follows
for $\LatCAs$: the minimum element is $\As=\CAs{2}$ and the maximum
element is $\Mag=\CAs{1}$. Algebraically, this says that any comb
associative operad projects onto $\As$ and is a quotient of~$\Mag$.
\medbreak

We end this section by relating the lattice $\LatQMag$ introduced in
Section~\ref{sec:Magmatic_operads} with the lattice $\LatCAs$. As
explained in Section~\ref{sec:operad_Mag}, a set-theoretic operad can
be embedded into a linear operad, so that the operads $\CAs{\gamma}$ can
be embedded into quotient operads $\KCAs{\gamma}$ of $\KMag$. Formally,
the operad $\KCAs{\gamma}$ is equal to $\KMag/_{I_{\gamma}}$, where
$I_{\gamma}$ is the operad ideal of $\KMag$ generated by
$\LComb{\gamma}-\RComb{\gamma}$. We obtain a new lattice $\LatKCAs$,
where $\KCAsAll$ is the set of all operads $\KCAs{\gamma}$. In this
linear framework, the condition $\KCAs{\gamma}\OrdCAs\KCAs{\gamma'}$
means that the dimension of the space
$\Hom\left(\KCAs{\gamma'},\KCAs{\gamma}\right)$ is equal to $1$. Hence,
$\LatKCAs$ is related to $\LatQMag$ by the following theorem.
\medbreak

\begin{Theorem} \label{thm:inclusion_lattice_CAs}
    The inclusion
    \begin{math}
        \iota:\left(\KCAsAll,\OrdCAs\right)
        \to
        \left(\QMag,\OrdQMag\right)
    \end{math}
    is nondecreasing. In particular, for all positive integers $\gamma$
    and $\gamma'$, we have
    \begin{equation} \label{equ:comparison_of_lattice_operations}
        \KCAs{\gcd \left(\bar{\gamma}, \bar{\gamma'}\right)}
        \OrdQMag
        \KCAs{\gamma}\InfQMag\KCAs{\gamma'}
    \end{equation}
    and
    \begin{equation}
        \KCAs{\gamma}\SupQMag\KCAs{\gamma'}
        \OrdQMag
        \KCAs{\lcm \left(\bar{\gamma}, \bar{\gamma'}\right)}.
    \end{equation}
\end{Theorem}
\medbreak

Note that $\LatKCAs$ does not embed as a sublattice of $\LatQMag$, that
is $\iota$ is not a lattice morphism. Consider for instance $\gamma = 3$
and $\gamma' = 4$, so that
\begin{equation}
    \KCAs{3} \InfCAs \KCAs{4} = \KCAs{2} = \K\Angle{\As},
\end{equation}
whereas
\begin{equation}
    \KCAs{3} \InfQMag \KCAs{4} =
    \KMag/_I
\end{equation}
where $I$ is the ideal of $\KMag$ generated by
$\LComb{3} - \RComb{3}$ and $\LComb{4} - \RComb{4}$.
\medbreak

\subsection{Completion of comb associative operads}
We are now looking for finite convergent presentations of comb
associative operads. By definition, the operad $\CAs{\gamma}$ is the
quotient of $\Mag$ by the operad congruence spanned by the rewrite rule
\begin{equation} \label{equ:rew_1}
    \LComb{\gamma}
    \enspace \Rew \enspace
    \RComb{\gamma}\,.
\end{equation}
This rewrite rule is compatible with the lexicographic order on prefix
words presented at the beginning of Section~\ref{sec:operad_Mag} in the
sense that the prefix word of the left member of~\eqref{equ:rew_1} is
lexicographically greater than the prefix word of the right one.
\medbreak

However, the rewrite relation $\RewContext$ induced by $\Rew$ is not
confluent for $\gamma\geq 3$. Indeed, we have
\begin{equation} \label{equ:branching_pair_CAs_3}
    \LComb{\gamma+1}
    \enspace \RewContext \enspace
    \begin{tikzpicture}[xscale=.22,yscale=.20,Centering]
        \node(0)at(0.00,-4.50){};
        \node(2)at(2.00,-4.50){};
        \node(6)at(6.00,-6.75){};
        \node(8)at(8.00,-6.75){};
        \node[NodeST](1)at(1.00,-2.25){\begin{math}\Product\end{math}};
        \node[NodeST](3)at(3.00,0.00){\begin{math}\Product\end{math}};
        \node[NodeST](5)at(5.30,-1.4){\begin{math}\gamma\end{math}};
        \node[NodeST](7)at(7.00,-4.50){\begin{math}\Product\end{math}};
        \draw[Edge](0)--(1);
        \draw[Edge](1)--(3);
        \draw[Edge](2)--(1);
        \draw[Edge](6)--(7);
        \draw[Edge, dotted](7)--(3);
        \draw[Edge](8)--(7);
        \node(r)at(3.00,1.74){};
        \draw[Edge](r)--(3);
    \end{tikzpicture}
    \qquad \mbox{and} \qquad
    \LComb{\gamma+1}
    \enspace \RewContext \enspace
    \begin{tikzpicture}[xscale=.22,yscale=.22,Centering]
        \node(0)at(0.00,-3.60){};
        \node(4)at(4.00,-7.20){};
        \node(6)at(6.00,-7.20){};
        \node(8)at(8.00,-1.80){};
        \node[NodeST](1)at(1.00,-1.80){\begin{math}\Product\end{math}};
        \node[NodeST](3)at(3.25,-2.90){\begin{math}\gamma\end{math}};
        \node[NodeST](5)at(5.00,-5.40){\begin{math}\Product\end{math}};
        \node[NodeST](7)at(7.00,0.00){\begin{math}\Product\end{math}};
        \draw[Edge](0)--(1);
        \draw[Edge](1)--(7);
        \draw[Edge,dotted](5)--(1);
        \draw[Edge](4)--(5);
        \draw[Edge](6)--(5);
        \draw[Edge](8)--(7);
        \node(r)at(7.00,1.5){};
        \draw[Edge](r)--(7);
    \end{tikzpicture}\,,
\end{equation}
and the two right members of~\eqref{equ:branching_pair_CAs_3} form a
branching pair which is not joinable (since these two trees are
normal forms of $\RewContext$).
\medbreak

In order to transform the rewrite relation induced by~\eqref{equ:rew_1}
into a convergent one, we apply the Buchberger algorithm for
operads~\cite[Section 3.7]{DK10} with respect to the lexicographic order
on prefix words. We first focus on the special case~$\gamma = 3$.
\medbreak

\subsubsection{The \texorpdfstring{$3$}{3}-comb associative operad}
\label{subsubsec:CAs_3}

The Buchberger algorithm applied on binary trees of degrees $4$ to $7$
provides the new rewrite rules \smallbreak
\input{Completion_CAs3}
\medbreak

Using prefix words, these new relations write as
\begin{multicols}{2}
\begin{small}
\begin{equation}
    22020200 \Rew 220020200,
\end{equation}
\begin{equation}
    20202202000 \Rew 20202020200,
\end{equation}
\begin{equation}
    20220200200 \Rew 20202020200,
\end{equation}
\begin{equation}
    2020220020200 \Rew 2020202022000,
\end{equation}
\begin{equation}
    2022002020200 \Rew 2020202200200,
\end{equation}

\begin{equation}
    2202200020200 \Rew 2200202020200,
\end{equation}
\begin{equation}
    202020202200200 \Rew 202020202022000,
\end{equation}
\begin{equation}
    202020220022000 \Rew 202020202020200,
\end{equation}
\begin{equation}
    220020202200200 \Rew 220020202022000,
\end{equation}
\begin{equation}
    220200202020200 \Rew 220020202022000.
\end{equation}
\end{small}
\end{multicols}
\medbreak

\begin{Theorem} \label{thm:convergent_rewrite_rule_CAs_3}
    The set $\Rew$ of rewrite rules containing~\eqref{equ:rew_1},
    and \eqref{equ:rew_2}---\eqref{equ:rew_11} is a finite
    convergent presentation of~$\CAs{3}$.
\end{Theorem}
\begin{proof}
    Let us show that the rewrite relation $\RewContext$ induced by
    $\Rew$ is convergent. First, for every relation $\Tfr \Rew \Tfr'$,
    we have $\Tfr > \Tfr'$. Therefore, by
    Lemma~\ref{lem:prefix_word_termination}, $\RewContext$ is
    terminating. Moreover, the greatest degree of a tree appearing in
    $\Rew$ is~$7$ so that, from Lemma~\ref{lem:degree_confluence}, to
    show that $\RewContext$ is convergent, it is enough to prove that
    each tree of degree at most $13$ admits exactly one normal form.
    Equivalently, this amounts to show that the number of normal forms
    of trees of arity $n\leq 14$ is equal to $\#\CAs{3}(n)$. By computer
    exploration, we get the same sequence
    \begin{equation} \label{equ:dimensions_CAs_3}
        1, 1, 2, 4, 8, 14, 20, 19, 16, 14, 14, 15, 16, 17
    \end{equation}
    for $\#\CAs{3}(n)$ and for the numbers of normal forms of arity $n$,
    when $ 1 \leq n \leq 14$, which proves the statement of the
    theorem.
\end{proof}
\medbreak

The rewrite rule $\Rew$ has, arity by arity, the cardinalities
\begin{equation}
    0, 0, 0, 1, 1, 2, 3, 4, 0, 0, \dots~.
\end{equation}
We also obtain from Theorem~\ref{thm:convergent_rewrite_rule_CAs_3} the
following consequences.
\medbreak

\begin{Proposition} \label{prop:PBW_basis_CAs_3}
    The set of the trees avoiding as subtrees the ones appearing as
    left members of $\Rew$ is a PBW basis of~$\CAs{3}$.
\end{Proposition}
\begin{proof}
    By definition of PBW bases and
    Theorem~\ref{thm:convergent_rewrite_rule_CAs_3}, the set
    $\NormalForms_{\RewContext}$ is a PBW basis of $\CAs{3}$ where
    $\RewContext$ is the rewrite relation induced by $\Rew$. Now, by
    Lemma~\ref{lem:normal_forms_avoiding}, $\NormalForms_{\RewContext}$
    can be described as the set of the trees avoiding the left members
    of~$\Rew$, whence the statement.
\end{proof}
\medbreak

\begin{Proposition} \label{prop:Hilbert_series_CAs_3}
    The Hilbert series of $\CAs{3}$ is
    \begin{equation} \label{equ:Hilbert_series_CAs_3}
        \HilbertSeries_{\CAs{3}}(t) = \frac{t}{(1 - t)^2}
        \left(1 - t + t^2 + t^3 + 2t^4 + 2t^5 - 7t^7 - 2t^8 + t^9 +
        2t^{10} + t^{11}\right).
    \end{equation}
\end{Proposition}
\begin{proof}
    From Proposition~\ref{prop:PBW_basis_CAs_3}, for any $n \geq 1$, the
    dimension of $\CAs{3}(n)$ is the number of trees that avoid as
    subtrees the left members of $\Rew$. Now, by using a result
    of~\cite{Gir18} (see also~\cite{Row10,KP15}) providing a system of
    equations for the generating series of the trees avoiding some sets
    of subtrees, we obtain Expression~\eqref{equ:Hilbert_series_CAs_3}
    for the considered family.
\end{proof}
\medbreak

For $n \leq 10$, the dimensions of $\CAs{3}(n)$ are provided by
Sequence~\eqref{equ:dimensions_CAs_3} and for all $n \geq 11$, the
Taylor expansion of~\eqref{equ:Hilbert_series_CAs_3} shows that
\begin{equation}
    \# \CAs{3}(n) = n + 3.
\end{equation}
\medbreak

Let us describe the elements of the PBW basis of $\CAs{3}$ for arity
$n \geq 11$. By Proposition~\ref{prop:PBW_basis_CAs_3}, these elements
are the normal forms of the rewrite relation induced by $\Rew$. Let for
any $d \geq 0$, the binary tree $\Zfr_d$ defined recursively by
\begin{equation}
    \Zfr_d :=
    \begin{cases}
        \Leaf & \mbox{if } d = 0, \\
        \Zfr_{d - 1}
        \circ_{\left\lfloor \frac{d-1}{2} \right\rfloor + 1} \Product
        & \mbox{otherwise}.
    \end{cases}
\end{equation}
For instance,
\begin{equation}
    \Zfr_4 =
    \begin{tikzpicture}[xscale=.22,yscale=.20,Centering]
        \node(0)at(0.00,-3.60){};
        \node(2)at(2.00,-7.20){};
        \node(4)at(4.00,-7.20){};
        \node(6)at(6.00,-5.40){};
        \node(8)at(8.00,-1.80){};
        \node[NodeST](1)at(1.00,-1.80){\begin{math}\Product\end{math}};
        \node[NodeST](3)at(3.00,-5.40){\begin{math}\Product\end{math}};
        \node[NodeST](5)at(5.00,-3.60){\begin{math}\Product\end{math}};
        \node[NodeST](7)at(7.00,0.00){\begin{math}\Product\end{math}};
        \draw[Edge](0)--(1);
        \draw[Edge](1)--(7);
        \draw[Edge](2)--(3);
        \draw[Edge](3)--(5);
        \draw[Edge](4)--(3);
        \draw[Edge](5)--(1);
        \draw[Edge](6)--(5);
        \draw[Edge](8)--(7);
        \node(r)at(7.00,1.35){};
        \draw[Edge](r)--(7);
    \end{tikzpicture}
    \qquad \mbox{and} \qquad
    \Zfr_5 =
    \begin{tikzpicture}[xscale=.22,yscale=.20,Centering]
        \node(0)at(0.00,-3.67){};
        \node(10)at(10.00,-1.83){};
        \node(2)at(2.00,-7.33){};
        \node(4)at(4.00,-9.17){};
        \node(6)at(6.00,-9.17){};
        \node(8)at(8.00,-5.50){};
        \node[NodeST](1)at(1.00,-1.83){\begin{math}\Product\end{math}};
        \node[NodeST](3)at(3.00,-5.50){\begin{math}\Product\end{math}};
        \node[NodeST](5)at(5.00,-7.33){\begin{math}\Product\end{math}};
        \node[NodeST](7)at(7.00,-3.67){\begin{math}\Product\end{math}};
        \node[NodeST](9)at(9.00,0.00){\begin{math}\Product\end{math}};
        \draw[Edge](0)--(1);
        \draw[Edge](1)--(9);
        \draw[Edge](10)--(9);
        \draw[Edge](2)--(3);
        \draw[Edge](3)--(7);
        \draw[Edge](4)--(5);
        \draw[Edge](5)--(3);
        \draw[Edge](6)--(5);
        \draw[Edge](7)--(1);
        \draw[Edge](8)--(7);
        \node(r)at(9.00,1.38){};
        \draw[Edge](r)--(9);
    \end{tikzpicture}.
\end{equation}
The normal forms split into two families. The first one is the set
of the $n - 1$ trees of the form
\begin{equation} \label{equ:normal_forms_type_A}
    \Zfr_d \circ_{d + 1} \Zfr_{n - 1 - d},
\end{equation}
for any $d \in [n - 1]$. For example, for $n = 12$,
\begin{equation}
    z_8 \circ_9 z_3 =
    \begin{tikzpicture}[xscale=.18,yscale=.14,Centering]
        \node(0)at(0.00,-5.56){};
        \node(10)at(10.00,-19.44){};
        \node(12)at(12.00,-13.89){};
        \node(14)at(14.00,-8.33){};
        \node(16)at(16.00,-8.33){};
        \node(18)at(18.00,-13.89){};
        \node(2)at(2.00,-11.11){};
        \node(20)at(20.00,-13.89){};
        \node(22)at(22.00,-11.11){};
        \node(24)at(24.00,-5.56){};
        \node(4)at(4.00,-16.67){};
        \node(6)at(6.00,-22.22){};
        \node(8)at(8.00,-22.22){};
        \node(19)at(19.00,-11){};
        \node[NodeST](1)at(1.00,-2.78){\begin{math}\Product\end{math}};
        \node[NodeST](11)at(11.00,-11.11)
            {\begin{math}\Product\end{math}};
        \node[NodeST](13)at(13.00,-5.56)
            {\begin{math}\Product\end{math}};
        \node[NodeST](15)at(15.00,0.00){\begin{math}\Product\end{math}};
        \node[NodeST](17)at(17.00,-5.56)
            {\begin{math}\Product\end{math}};
        \node[NodeST](21)at(21.00,-8.33)
            {\begin{math}\Product\end{math}};
        \node[NodeST](23)at(23.00,-2.78)
            {\begin{math}\Product\end{math}};
        \node[NodeST](3)at(3.00,-8.33){\begin{math}\Product\end{math}};
        \node[NodeST](5)at(5.00,-13.89){\begin{math}\Product\end{math}};
        \node[NodeST](7)at(7.00,-19.44){\begin{math}\Product\end{math}};
        \node[NodeST](9)at(9.00,-16.67){\begin{math}\Product\end{math}};
        \draw[Edge](0)--(1);
        \draw[Edge](1)--(15);
        \draw[Edge](10)--(9);
        \draw[Edge](11)--(3);
        \draw[Edge](12)--(11);
        \draw[Edge](13)--(1);
        \draw[Edge](14)--(13);
        \draw[Edge](16)--(17);
        \draw[Edge](17)--(23);
        \draw[Edge](19)--(21);
        \draw[Edge](2)--(3);
        \draw[Edge](21)--(17);
        \draw[Edge](22)--(21);
        \draw[Edge](23)--(15);
        \draw[Edge](24)--(23);
        \draw[Edge](3)--(13);
        \draw[Edge](4)--(5);
        \draw[Edge](5)--(11);
        \draw[Edge](6)--(7);
        \draw[Edge](7)--(9);
        \draw[Edge](8)--(7);
        \draw[Edge](9)--(5);
        \node(r)at(15.00,2.5){};
        \draw[Edge](r)--(15);
    \end{tikzpicture}
\end{equation}
is a tree of this first family. The second family contains the
following four trees
\begin{equation}\label{equ:normal_forms_type_B}
    \RComb{n-1}, \qquad
    \begin{tikzpicture}[xscale=.22,yscale=.20,Centering]
        \node(0)at(0.00,-4.50){};
        \node(2)at(2.00,-4.50){};
        \node(8)at(8.00,-6.75){};
        \node(10)at(5,-9){};
        \node(11)at(7,-9){};
        \node[NodeST](1)at(1.00,-2.25){\begin{math}\Product\end{math}};
        \node[NodeST](3)at(3.00,0.00){\begin{math}\Product\end{math}};
        \node[NodeST](5)at(6.5,-1.4){\begin{math}n-2\end{math}};
        \node[NodeST](7)at(7.00,-4.50){\begin{math}\Product\end{math}};
        \node(9)at(6,-6.75){};
        \draw[Edge](0)--(1);
        \draw[Edge](1)--(3);
        \draw[Edge](2)--(1);
        \draw[Edge](7)--(9);
        \draw[Edge, dotted](7)--(3);
        \draw[Edge](8)--(7);
        \node(r)at(3.00,1.74){};
        \draw[Edge](r)--(3);
    \end{tikzpicture}\,, \qquad
    \begin{tikzpicture}[xscale=.22,yscale=.20,Centering]
        \node(0)at(0.00,-4.50){};
        \node(2)at(2.00,-4.50){};
        \node(8)at(8.00,-6.75){};
        \node(10)at(5,-9){};
        \node(11)at(7,-9){};
        \node(1)at(1.00,-2.25){};
        \node[NodeST](3)at(3.00,0.00){\begin{math}\Product\end{math}};
        \node[NodeST](5)at(6.5,-1.4){\begin{math}n-2\end{math}};
        \node[NodeST](7)at(7.00,-4.50){\begin{math}\Product\end{math}};
        \node[NodeST](9)at(6,-6.75){\begin{math}\Product\end{math}};
        \draw[Edge](1)--(3);
        \draw[Edge](7)--(9);
        \draw[Edge](9)--(10);
        \draw[Edge](9)--(11);
        \draw[Edge, dotted](7)--(3);
        \draw[Edge](8)--(7);
        \node(r)at(3.00,1.74){};
        \draw[Edge](r)--(3);
    \end{tikzpicture}\,, \qquad
    \begin{tikzpicture}[xscale=.22,yscale=.20,Centering]
        \node(0)at(0.00,-4.50){};
        \node(2)at(2.00,-4.50){};
        \node(8)at(8.00,-6.75){};
        \node(10)at(5,-9){};
        \node(11)at(7,-9){};
        \node[NodeST](1)at(1.00,-2.25){\begin{math}\Product\end{math}};
        \node[NodeST](3)at(3.00,0.00){\begin{math}\Product\end{math}};
        \node[NodeST](5)at(6.5,-1.4){\begin{math}n-3\end{math}};
        \node[NodeST](7)at(7.00,-4.50){\begin{math}\Product\end{math}};
        \node[NodeST](9)at(6,-6.75){\begin{math}\Product\end{math}};
        \draw[Edge](0)--(1);
        \draw[Edge](1)--(3);
        \draw[Edge](2)--(1);
        \draw[Edge](7)--(9);
        \draw[Edge](9)--(10);
        \draw[Edge](9)--(11);
        \draw[Edge, dotted](7)--(3);
        \draw[Edge](8)--(7);
        \node(r)at(3.00,1.74){};
        \draw[Edge](r)--(3);
    \end{tikzpicture}.
\end{equation}
Observe that each of these $n + 3$ trees avoids the left members of the
convergent presentation given in
Theorem~\ref{thm:convergent_rewrite_rule_CAs_3}. Moreover, we have the
following.
\medbreak

\begin{Proposition} \label{prop:absorbing_normal_forms_type_B}
    For any two trees $\Tfr$ and $\Tfr'$, if one of them is of the
    form~\eqref{equ:normal_forms_type_B} and $i \in [|\Tfr|]$, then
    the normal form for the rewrite relation induced by $\Rew$ of
    $\Tfr \circ_i \Tfr'$ is a tree of the
    form~\eqref{equ:normal_forms_type_B}.
\end{Proposition}
\begin{proof}
    None of the trees of the form~\eqref{equ:normal_forms_type_A}
    contain the left or the right member of the rewrite
    rule~\eqref{equ:rew_1} generating the congruence relation
    $\CongrCAs{3}$. Hence, these trees are alone in their equivalence
    classes. Moreover, the trees of the
    form~\eqref{equ:normal_forms_type_B} contain the right member of
    the rewrite rule~\eqref{equ:rew_1} generating the congruence
    relation $\CongrCAs{3}$, so that the composition of such a tree
    with another tree is not alone in its equivalence class, and thus
    it does not belong to the family~\eqref{equ:normal_forms_type_A}.
\end{proof}
\medbreak

Proposition~\ref{prop:absorbing_normal_forms_type_B} says that the
family of trees~\eqref{equ:normal_forms_type_B} is absorbing
for the partial composition.
\medbreak

Computer explorations allow us to conjecture the multiplication table
of the exhibited PBW basis of $\CAs{3}$. However, we do note have a
simple description of this table. For instance, the partial composition
$\Tfr\circ_i\Tfr'$ where $\Tfr$ is a tree of the
form~\eqref{equ:normal_forms_type_A} and $\Tfr'$ is a tree of the
form~\eqref{equ:normal_forms_type_B} can be fully described by $36$
cases depending on the values of various parameters associated with
$\Tfr$, $\Tfr'$, and~$i$.
\medbreak

\subsubsection{Higher comb associative operads}
\label{sec:higher_comb_associative_operads}

We run the same algorithm for $\CAs{\gamma}$ when $\gamma \in [9]$.
Table~\ref{tab:cardinalities_completion_CAs} shows
the number of rewrite rules needed to obtain complete orientations.
\medbreak

\begin{table}[ht]
    \centering
    \footnotesize
    \setlength{\tabcolsep}{.35em}
    \begin{tabular}{c||ccccccccccccccccccccccccccc}
        $\gamma$ & \multicolumn{27}{c}{
        Cardinalities of completions of $\CAs{\gamma}$}
        \\ \hline \hline
        1 & 0 & 0 & 0 & 0 & 0 & 0 & 0 & 0 & 0 & 0 & 0 & 0 & 0 & 0 & 0
        & 0 & 0 & 0 & 0 & 0 & 0 & 0 & 0 & 0 & 0 & 0 & 0
        \\
        2 & 0 & 0 & 1 & 0 & 0 & 0 & 0 & 0 & 0 & 0 & 0 & 0 & 0 & 0 & 0
        & 0 & 0 & 0 & 0 & 0 & 0 & 0 & 0 & 0 & 0 & 0 & 0
        \\
        3 & 0 & 0 & 0 & 1 & 1 & 2 & 3 & 4 & 0 & 0 & 0 & 0 & 0 & 0 & 0
        & 0 & 0 & 0 & 0 & 0 & 0 & 0 & 0 & 0 & 0 & 0 & 0
        \\
        4 & 0 & 0 & 0 & 0 & 1 & 1 & 0 & 3 & 4 & 5 & 18 & 22 & 11 & 12
        & 15 & 19 & 25 & 36 & 44 & 52 & 68 & 79 & 93 & 105 & 106 & 109
        & 107
        \\
        5 & 0 & 0 & 0 & 0 & 0 & 1 & 1 & 0 & 0 & 4 & 5 & 8 & 18 & 31
        & 36 & 48 & 73 & 111 & 172 & 272 & 455 & 783
        \\
        6 & 0 & 0 & 0 & 0 & 0 & 0 & 1 & 1 & 0 & 0 & 0 & 5 & 6 & 11
        & 23 & 30 & 48 & 73 & 117 & 204 & 348 & 589 & 1004
        \\
        7 & 0 & 0 & 0 & 0 & 0 & 0 & 0 & 1 & 1 & 0 & 0 & 0 & 0 & 6 & 7
        & 16 & 24 & 32 & 49 & 88 & 150 & 261 & 475 & 854
        \\
        8 & 0 & 0 & 0 & 0 & 0 & 0 & 0 & 0 & 1 & 1 & 0 & 0 & 0 & 0 & 0
        & 7 & 8 & 21 & 29 & 34 & 53 & 93 & 172 & 311 & 565
        \\
        9 & 0 & 0 & 0 & 0 & 0 & 0 & 0 & 0 & 0 & 1 & 1 & 0 & 0 & 0 & 0
        & 0 & 0 & 8 & 9 & 28 & 30 & 36 & 57 & 101 & 185 & 348 & 648
    \end{tabular}
    \caption{\footnotesize
    The sequences of the cardinalities, arity by arity, of the rewrite
    rules being completions of orientations of~$\CongrCAs{\gamma}$.}
    \label{tab:cardinalities_completion_CAs}
\end{table}

From these computer explorations, we conjecture that the algorithm we
use does not provide a finite convergent presentation of $\CAs{\gamma}$,
when $\gamma \geq 4$. We point out that for $\CAs{4}$ new rewrite
rules still appear in arity $42$. Moreover, the total number of rewrite
rules at this arity is~$3149$. Our program was too slow to
compute further the completions of $\CAs{5}$, $\CAs{6}$, $\CAs{7}$, and
$\CAs{8}$.
\medbreak

However, the completion algorithm we use depends on the chosen order on
the trees. In order to find out if the completion algorithm leads to a
finite convergent presentation using a different order, we run the
following backtracking algorithm. For every branching pair
$\left\{\Tfr_1, \Tfr_2\right\}$ which is not joinable, we recursively
try to find a completion of $\Rew$ by adding either the rule
$\Tfr_1 \Rew  \Tfr_2$ or $\Tfr_2 \Rew  \Tfr_1$.
If at any moment the rewrite relation $\RewContext$ induced by~$\Rew$
loops (that is $\RewContext$ is not antisymmetric), we simply reject it.
We do not find any finite presentation for $\CAs{4}$, $\CAs{5}$, and
$\CAs{6}$ until arity $12$. We conjecture that there is no finite
convergent presentation of $\CAs{\gamma}$ when $\gamma \geq 4$ and when
the left and the right members of the rewrite rules are trees
belonging to $\Mag$.
\medbreak

Thanks to the partial completions presented in
Table~\ref{tab:cardinalities_completion_CAs}, we can compute the
following first dimensions of $\CAs{\gamma}$.
Table~\ref{tab:dimensions_CAs} shows the first dimensions of the
operads $\CAs{\gamma}$ for $\gamma \in [9]$.
\begin{table}[ht]
    \centering
    \footnotesize
    \setlength{\tabcolsep}{.35em}
    \begin{tabular}{c||ccccccccccccccccc}
        $\gamma$ & \multicolumn{17}{c}{Dimensions of $\CAs{\gamma}$}
        \\ \hline \hline
        1 & 1 & 1 & 2 & 5 & 14 & 42 & 132 & 429 & 1430 & 4862
        & 16796 & 58786 & 208012 & 742900 & 2674440 & 9694845
        & 35357670
        \\
        2 & 1 & 1 & 1 & 1 & 1 & 1 & 1 & 1 & 1 & 1 & 1 & 1 & 1 & 1 & 1
        & 1 & 1
        \\
        3 & 1 & 1 & 2 & 4 & 8 & 14 & 20 & 19 & 16 & 14 &
        14 & 15 & 16 & 17 & 18 & 19 & 20
        \\
        4 & 1 & 1 & 2 & 5 & 13 & 35 & 96 & 264 & 724 & 1973 & 5335
        & 14390 & 38872 & 105141 & 284929 & 774254 & 2111088
        \\
        5 & 1 & 1 & 2 & 5 & 14 & 41 & 124 & 384 & 1210 & 3861 & 12440
        & 40392 & 131997 & 433782 & 1432696 & 4752857
        & 15829261
        \\
        6 & 1 & 1 & 2 & 5 & 14 & 42 & 131 & 420 & 1375 & 4576 & 15431
        & 52598 & 180895 & 626862 & 2186504 & 7670138
        & 27041833
        \\
        7 & 1 & 1 & 2 & 5 & 14 & 42 & 132 & 428 & 1420 & 4796 & 16432
        & 56966 & 199444 & 704140 & 2503914 & 8959699
        & 32236657
        \\
        8 & 1 & 1 & 2 & 5 & 14 & 42 & 132 & 429 & 1429 & 4851 & 16718
        & 58331 & 205632 & 731272 & 2620176 & 9449688 & 34276116
        \\
        9 & 1 & 1 & 2 & 5 & 14 & 42 & 132 & 429 & 1430 & 4861 & 16784
        & 58695 & 207452 & 739840 & 2658936 & 9620232 & 35011566
    \end{tabular}
    \medbreak

    \caption{\footnotesize
    The sequences, arity by arity, of the dimensions of~$\CAs{\gamma}$.}
    \label{tab:dimensions_CAs}
\end{table}
\medbreak

\section{Equating two cubic trees} \label{sec:MAg_3}
In this section, we explore all the quotients of $\Mag$ obtained by
equating two trees of degree $3$. We denote by $\Afr_i$ the $i$th
cubic tree for the lexicographic order, that is
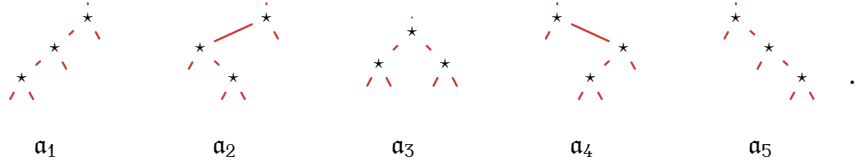
\begin{center}
    \begin{tabular}{ccccc}
        \quad $\TreeA$ \quad & \quad $\TreeB$ \quad
        & \quad $\TreeC$ \quad & \quad $\TreeD$ \quad
        & \quad $\TreeE$ \quad \\
        $\Afr_1$ & $\Afr_2$ & $\Afr_3$ & $\Afr_4$ & $\Afr_5$
    \end{tabular}.
\end{center}
We denote by $\Mag^{\{i, j\}}$ the quotient operad $\Mag/_{\Congr}$,
where $\Congr$ is the operad congruence generated by
$\Afr_i \Congr \Afr_j$. We have already studied the operad
$\Mag^{\{1, 5\}} = \CAs{3}$ in Section~\ref{subsubsec:CAs_3}.
\medbreak

\subsection{Anti-isomorphic classes of quotients}
Some of the quotients $\Mag^{\{i, j\}}$ are anti\--iso\-morphic one of
the other. Indeed, the map $\phi : \Mag \to \Mag$ sending any binary
tree $\Tfr$ to the binary tree obtained by exchanging recursively the
left and right subtrees of $\Tfr$ is an anti-isomorphism of $\Mag$. For
this reason, the $\binom{5}{2} = 10$ quotients $\Mag^{\{i, j\}}$ of
$\Mag$ fit into the six equivalence classes
\begin{multline}
    \left\{\Mag^{\{1, 2\}}, \Mag^{\{4, 5\}}\right\},
    \left\{\Mag^{\{1, 3\}}, \Mag^{\{3, 5\}}\right\},
    \left\{\Mag^{\{1, 4\}}, \Mag^{\{2, 5\}}\right\}, \\
    \left\{\CAs{3}\right\},
    \left\{\Mag^{\{2, 3\}}, \Mag^{\{3, 4\}}\right\},
    \left\{\Mag^{\{2, 4\}}\right\}
\end{multline}
of anti-isomorphic operads.
\medbreak

Given an operad $\Oca$ with partial compositions $\circ_i$, we consider
the partial compositions $\bar{\circ}_i$ defined by
\begin{math}
  x \, \bar{\circ}_i \, y := x \circ_{|x| - i + 1} y
\end{math}
for all $x, y \in \Oca$ and $i \in [|x|]$. The reader can easily check
the assertions of the following lemma.
\medbreak

\begin{Lemma} \label{lem:recall_anti_isomorphic}
    Let $\Oca_1$ and $\Oca_2$ be two anti-isomorphic operads and let
    $\phi$ be an anti-isomorphism between $\Oca_1$ and $\Oca_2$. Then,
    \begin{enumerate}[label={(\it\roman*)}]
        \item \label{item:recall_anti_isomorphic_1}
        $\HilbertSeries_{\Oca_1}(t) = \HilbertSeries_{\Oca_2}(t)$;
        \item \label{item:recall_anti_isomorphic_2}
        if $\Rew^{(1)}$ is a convergent presentation of $\Oca_1$, then
        the set of rewrite rules $\Rew^{(2)}$ satisfying
        $\phi(x) \Rew^{(2)} \phi(y)$ for any $x, y \in \Oca_1$ whenever
        $x \Rew^{(1)} y$, is a convergent presentation of~$\Oca_2$;
        \item\label{item:realization_anti_isomorphic}
        If $\left(\Oca, \circ_i\right)$ is a combinatorial realization
        of $\Oca_1$, then
        \begin{math}
          \left(\Oca, \bar{\circ}_i\right)
        \end{math}
        is a combinatorial realization of $\Oca_2$.
    \end{enumerate}
\end{Lemma}
\medbreak

\subsection{Quotients on integer compositions}
Four among the six equivalence classes of the quotients
$\Mag^{\{i, j\}}$ of $\Mag$ can be realized in terms of operads on
integer compositions. Let us review these.
\medbreak

\subsubsection{Operads $\Mag^{\{1, 2\}}$ and $\Mag^{\{4, 5\}}$}
\label{subsubsec:Mag_1_2}
The reader can check, using the Buchberger algorithm for operads,
that the rewrite rule $\Afr_2 \Rew \Afr_1$ is a
convergent presentation of $\Mag^{\{1, 2\}}$. The operads
$\Mag^{\{1, 2\}}$ and $\Mag^{\{4, 5\}}$ are anti-isomorphic, so that
by Lemma~\ref{lem:recall_anti_isomorphic}, the rewrite rule
$\Afr_4 \Rew \Afr_5$ is a convergent presentation of~$\Mag^{\{4, 5\}}$.
In a similar fashion as Proposition~\ref{prop:Hilbert_series_CAs_3}, we
compute the following result thanks to~\cite{Gir18}.
\medbreak

\begin{Theorem} \label{thm:Hilbert_series_Mag_1_2}
    The Hilbert series of $\Mag^{\{1, 2\}}$ and $\Mag^{\{4, 5\}}$ are
    \begin{equation} \label{equ:Hilbert_series_Mag_1_2}
        \HilbertSeries_{\Mag^{\{1, 2\}}}(t)
        = \HilbertSeries_{\Mag^{\{4, 5\}}}(t) =
        t \frac{1 - t}{1 - 2t}.
    \end{equation}
\end{Theorem}
\medbreak

A Taylor expansion of series~\eqref{equ:Hilbert_series_Mag_1_2} shows
the following.
\medbreak

\begin{Proposition} \label{prop:Close_formula_Mag_1_2}
    For all $n \geq 2,$
    \begin{equation} \label{equ:Close_formula_Mag_1_2}
        \# \Mag^{\{1, 2\}}(n) = \# \Mag^{\{4, 5\}}(n) = 2^{n-2}.
    \end{equation}
\end{Proposition}
\medbreak

Many graduate sets of combinatorial objects are enumerated by powers of
$2$. We choose to present a combinatorial realization of
$\Mag^{\{1, 2\}}$ based on integer compositions. Recall that an
\Def{integer composition} is a finite sequence of positive integers. If
$\LambdaB := \left(\LambdaB_1, \dots, \LambdaB_p\right)$ is an integer
composition, we denote by $s_{i, j}(\LambdaB)$ the number
$1 + \sum_{i \leq k \leq j} \LambdaB_k$. The \Def{arity} of $\LambdaB$
is $s_{1, p}(\LambdaB)$. Observe that the empty integer composition
$\epsilon$ is the unique object of arity $1$. The graded set of all
integer compositions is denoted by~$\Compositions$.
\medbreak

Given an integer $i \geq 1$, we define the binary operation
\begin{math}
    \circ_i^{(1,2)} : \Compositions(n) \times \Compositions(m)
    \to\Compositions(n + m - 1)
\end{math}
for any integer compositions
$\LambdaB := \left(\LambdaB_1, \dots, \LambdaB_p\right)$ and
$\MuB := \left(\MuB_1, \dots, \MuB_q\right)$ of respective arities
$n \geq i$ and $m \geq 1$ by
\begin{equation}
    \LambdaB \circ_i^{(1,2)} \MuB :=
    \begin{cases}
        \left(\LambdaB_1, \dots, \LambdaB_p,
        \MuB_1, \dots, \MuB_q\right) &
        \mbox{if } i = n, \\
        \left(\LambdaB_1, \dots, \LambdaB_k, \LambdaB_{k + 1} + m - 1,
        \LambdaB_{k + 2}, \dots, \LambdaB_p\right)
            & \mbox{otherwise},
    \end{cases}
\end{equation}
where $k \geq 0$ is such that
\begin{math}
    s_{1, k}(\LambdaB) \leq i < s_{1, k + 1}(\LambdaB).
\end{math}
\medbreak

\begin{Proposition} \label{prop:Realisation_Mag_1_2}
    The operad $\left(\Compositions, \circ_i^{(1,2)}\right)$ is a
    combinatorial realization of $\Mag^{\{1, 2\}}$.
\end{Proposition}
\begin{proof}
    We have to show that the operads
    $\left(\Compositions, \circ_i^{(1,2)}\right)$ and $\Mag^{\{1, 2\}}$
    are isomorphic. The set of normal forms of arity $n$ for the rewrite
    relation $\RewContext$ induced by the rule $\Afr_2 \Rew \Afr_1$ is
    \begin{equation} \label{equ:Normal_forms_A2}
        \left\{\LComb{\LambdaB_1, \dots, \LambdaB_p} :
        \left(\LambdaB_1, \dots, \LambdaB_p\right)
        \in \Compositions(n) \right\}
    \end{equation}
    where
    \begin{equation}
        \LComb{\LambdaB_1, \dots, \LambdaB_p} :=
        \RComb{p} \circ
        \left[\LComb{\LambdaB_1-1}, \dots, \LComb{\LambdaB_p-1},
        \Leaf\right].
    \end{equation}
    Thus, the map $\phi : \Mag^{\{1, 2\}}(n) \to \Compositions(n)$
    defined by
    \begin{equation} \label{equ:Morphism_A2}
        \phi\left(\LComb{\LambdaB_1, \dots, \LambdaB_p} \right) :=
        \left(\LambdaB_1, \dots, \LambdaB_p\right)
    \end{equation}
    is a bijection. Let us show that $\phi$ is an operad morphism.
    Let $\LambdaB := \left(\LambdaB_1, \dots, \LambdaB_p\right)$ and
    $\MuB := \left(\MuB_1, \dots, \MuB_q\right)$ be two integer
    compositions of respective arities $n$ and $m$, and let
    \begin{equation}
        \Tfr := \LComb{\LambdaB_1, \dots, \LambdaB_p}
        \in \Mag^{\{1, 2\}}(n)
        \quad \mbox{ and } \quad
        \Tfr' := \LComb{\MuB_1, \dots, \MuB_q} \in \Mag^{\{1, 2\}}(m).
    \end{equation}
    The tree $\Tfr \circ_n \Tfr'$ is equal to
    $\LComb{\LambdaB_1, \dots, \LambdaB_p, \MuB_1, \dots, \MuB_q}$, so
    that
    \begin{math}
        \phi\left(\Tfr \circ_n \Tfr'\right)
        = \phi\left(\Tfr\right) \circ_n^{(1,2)} \phi\left(\Tfr'\right).
    \end{math}
    Let $i \in [n - 1]$ and $k$ be such that
    $s_{1, k}(\LambdaB) \leq i < s_{1, k+1}(\LambdaB)$, so that
    $\Tfr \circ_i \Tfr'$ is equal to
    \begin{equation} \label{equ:Rewriting_1_Mag_1_2}
        \RComb{p} \circ \left[\LComb{\LambdaB_1-1}, \dots,
            \LComb{\LambdaB_k-1}, \LComb{\LambdaB_{k+1}-1}
            \circ_{i+1-s_{1,k}(\LambdaB)}\Tfr', \LComb{\LambdaB_{k+2}-1},
            \dots, \LComb{\LambdaB_p-1}, \Leaf \right].
    \end{equation}
    The tree~\eqref{equ:Rewriting_1_Mag_1_2} rewrites by $\RewContext$
    into
    \begin{equation}
        \RComb{p} \circ \left[\LComb{\LambdaB_1-1}, \dots,
        \left(
        \LComb{\LambdaB_{k+1}-1} \circ_{i+1-s_{1,k}(\LambdaB)}
        \LComb{\MuB_1-1}
        \right)
        \circ_{i+1-s_{1,k}(\LambdaB)+\MuB_1}
        \LComb{\MuB_2, \dots, \MuB_q},
        \dots, \LComb{\LambdaB_p-1},
        \Leaf \right]
    \end{equation}
    which rewrite itself by $\RewContext$ into
    \begin{equation} \label{equ:Rewriting_2_Mag_1_2}
        \RComb{p} \circ \left[\LComb{\LambdaB_1-1}, \dots,
        \LComb{\LambdaB_k-1}, \LComb{\LambdaB_{k+1}-1+\MuB_1}
        \circ_{i+1-s_{1,k}(\LambdaB)+\MuB_1}
        \LComb{\MuB_2, \dots, \MuB_q}, \LComb{\LambdaB_{k+2}-1}, \dots,
        \LComb{\LambdaB_p-1}, \Leaf \right]
    \end{equation}
    in $\MuB_1-1$ steps. By iterating $q-1$ times the rewrite steps
    passing from~\eqref{equ:Rewriting_1_Mag_1_2}
    to~\eqref{equ:Rewriting_2_Mag_1_2}, we get
    \begin{equation}
        \Tfr \circ_i \Tfr'
        \RewContext
        \LComb{\LambdaB_1, \dots, \LambdaB_k,
        \LambdaB_{k+1} + m-1, \LambdaB_{k+2}, \dots, \LambdaB_p},
    \end{equation}
    so that
    \begin{math}
        \phi(\Tfr \circ_i \Tfr')
        = \phi(\Tfr) \circ_i^{(1,2)} \phi(\Tfr').
    \end{math}
    Therefore $\phi$ is an operad morphism.
\end{proof}
\medbreak

From Lemma~\ref{lem:recall_anti_isomorphic}, we deduce that
$\left(\Compositions, \bar{\circ}_i^{(1,2)}\right)$ is a combinatorial
realization of $\Mag^{\{4, 5\}}$.
\medbreak

\subsubsection{Operads $\Mag^{\{1, 3\}}$ and $\Mag^{\{3, 5\}}$}
The reader can check that the rewrite rule $\Afr_3 \Rew \Afr_1$ is a
convergent presentation of $\Mag^{\{1, 3\}}$. By
Lemma~\ref{lem:recall_anti_isomorphic}, the rewrite rule
$\Afr_3 \Rew \Afr_5$ is a convergent presentation of~$\Mag^{\{3, 5\}}$.
Thanks to~\cite{Gir18}, the Hilbert series of $\Mag^{\{1, 3\}}$ and
$\Mag^{\{3, 5\}}$ are equals to~\eqref{equ:Hilbert_series_Mag_1_2}. Thus,
for all $n \geq 2$,
$\# \Mag^{\{1, 3\}}(n)$ and $\# \Mag^{\{3, 5\}}(n)$ are
equal to~\eqref{prop:Close_formula_Mag_1_2}.
\medbreak

Like in Section~\ref{subsubsec:Mag_1_2}, we choose a combinatorial
realization based on integer compositions.
Given an integer $i \geq 1$, we define the binary operation
\begin{math}
    \circ_i^{(1,3)} : \Compositions(n) \times \Compositions(m)
    \to\Compositions(n + m - 1)
\end{math}
for any integer compositions
$\LambdaB := \left(\LambdaB_1, \dots, \LambdaB_p\right)$ and
$\MuB := \left(\MuB_1, \dots, \MuB_q\right)$ of respective arities
$n \geq i$ and $m \geq 1$ by
\begin{equation}
    \LambdaB \circ_i^{(1,3)} \MuB :=
    \begin{cases}
        \left(\LambdaB_1, \dots, \LambdaB_{i - 1},
        \MuB_1 + s_{i, p}(\LambdaB), \MuB_2, \dots, \MuB_q\right) &
        \mbox{if } i \leq p + 1, \\
        \left(\LambdaB_1, \dots, \LambdaB_{k - 1},
        \LambdaB_k + m - 1,
        \LambdaB_{k + 1}, \dots, \LambdaB_p\right)
            & \mbox{otherwise},
    \end{cases}
\end{equation}
where $k \geq 0$ is such that
\begin{math}
    k + 1 + s_{k + 1, p}(\LambdaB) \leq i < k + s_{k, p}(\LambdaB).
\end{math}
\medbreak

\begin{Proposition} \label{prop:Realisation_Mag_1_3}
    The operad $\left(\Compositions, \circ_i^{(1,3)}\right)$ is a
    combinatorial realization of $\Mag^{\{1, 3\}}$.
\end{Proposition}
\begin{proof}
    The proof is similar to the one of
    Proposition~\ref{prop:Realisation_Mag_1_2}, thus we just give an
    outline of it. We have to show that the operads
    $\left(\Compositions, \circ_i^{(1,3)}\right)$ and
    $\Mag^{\{1, 3\}}$ are isomorphic.
    The set of normal forms of arity $n$ for the rewrite rule
    $\RewContext$ induced by
    $\Afr_3 \Rew \Afr_1$ is
    \begin{equation}
      \left\{\Lightning{\LambdaB_1, \dots, \LambdaB_p} :
      \left(\LambdaB_1, \dots, \LambdaB_p\right)
      \in \Compositions(n) \right\}
    \end{equation}
    where
    \begin{equation}
        \Lightning{\LambdaB_1, \dots, \LambdaB_p} :=
        \LComb{\LambdaB_1} \circ_2 \left(\LComb{\LambdaB_2}
        \circ_2 \left(\cdots \left(
        \LComb{\LambdaB_{p-1}} \circ_2 \LComb{\LambdaB_p} \right)
        \cdots \right) \right).
    \end{equation}
    Thus, it is possible to show that the map
    $\phi : \Mag^{\{1, 3\}}(n) \to \Compositions(n)$ defined by
    \begin{equation}
        \phi\left(\Lightning{\LambdaB_1, \dots, \LambdaB_p} \right)
        :=
        \left(\LambdaB_1, \dots, \LambdaB_p \right)
    \end{equation}
    is an operad isomorphism from $\Mag^{\{1, 3\}}$ to
    $\left(\Compositions, \circ_i^{(1,3)}\right)$.
\end{proof}
\medbreak

From Lemma~\ref{lem:recall_anti_isomorphic}, we deduce that
$\left(\Compositions, \bar{\circ}_i^{(1,3)}\right)$ is a combinatorial
realization of $\Mag^{\{3, 5\}}$.
\medbreak

\subsubsection{Operads $\Mag^{\{1, 4\}}$ and $\Mag^{\{2, 5\}}$}
The reader can check that the rewrite rule $\Afr_4 \Rew \Afr_1$ is a
convergent presentation of $\Mag^{\{1, 4\}}$. By
Lemma~\ref{lem:recall_anti_isomorphic}, the rewrite rule
$\Afr_2 \Rew \Afr_5$ is a convergent presentation of~$\Mag^{\{2, 5\}}$.
Thanks to~\cite{Gir18}, the Hilbert series of $\Mag^{\{1, 4\}}$ and
$\Mag^{\{2, 5\}}$ are equals to~\eqref{equ:Hilbert_series_Mag_1_2}. Thus,
for $n \geq 2$, $\# \Mag^{\{1, 4\}}(n)$ and $\# \Mag^{\{2, 5\}}(n)$ are
equal to~\eqref{prop:Close_formula_Mag_1_2}.
\medbreak

Like in Section~\ref{subsubsec:Mag_1_2}, we choose a combinatorial
realization based on integer compositions. Given an integer $i \geq 1$,
we define the binary operation
\begin{math}
    \circ_i^{(2,5)} : \Compositions(n) \times \Compositions(m)
    \to\Compositions(n + m - 1)
\end{math}
for any integer compositions
$\LambdaB := \left(\LambdaB_1, \dots, \LambdaB_p\right)$ and
$\MuB := \left(\MuB_1, \dots, \MuB_q\right)$ of respective arities
$n \geq i$ and $m \geq 1$ by
\begin{equation}
    \LambdaB \circ_i^{(2,5)} \MuB :=
    \begin{cases}
        \left(\LambdaB_1, \dots, \LambdaB_k,
        \MuB_1, \dots, \MuB_{q-1}, \MuB_q+
        \LambdaB_{k + 1}, \dots, \LambdaB_p\right)
            & \mbox{if } i = s_{1, k}(\LambdaB), \\
        \left(\LambdaB_1, \dots, \LambdaB_k,
        i - s_{1, k}(\LambdaB), \MuB_1, \dots,
        \MuB_q, s_{1, k+1}(\LambdaB) - i, \LambdaB_{k + 2}, \dots,
        \LambdaB_p
        \right)
            & \mbox{otherwise},
    \end{cases}
\end{equation}
where $k \geq 0$ is such that
\begin{math}
    s_{1, k}(\LambdaB) \leq i < s_{1, k + 1}(\LambdaB).
\end{math}
\medbreak

\begin{Proposition} \label{prop:Realisation_Mag_1_4}
    The operad $\left(\Compositions, \circ_i^{(2,5)}\right)$ is a
    combinatorial realization of~$\Mag^{\{2, 5\}}$.
\end{Proposition}
\begin{proof}
    The proof is similar to the one of
    Proposition~\ref{prop:Realisation_Mag_1_2}, thus we just give an
    outline of it.
    We have to show that the operad
    $\left(\Compositions, \circ_i^{(2,5)}\right)$ and $\Mag^{\{2, 5\}}$
    are isomorphic. The set of normal forms of arity $n$ for the
    rewrite rule $\RewContext$ induced by $\Afr_2 \Rew \Afr_5$
    is~\eqref{equ:Normal_forms_A2}. Thus, it is possible to show that
    the map~\eqref{equ:Morphism_A2} is an operad isomorphism from
    $\Mag^{\{2, 5\}}$ to~$\left(\Compositions, \circ_i^{(2,5)}\right)$.
\end{proof}
\medbreak

From Lemma~\ref{lem:recall_anti_isomorphic}, we deduce that
$\left(\Compositions, \bar{\circ}_i^{(2,5)}\right)$ is a combinatorial
realization of $\Mag^{\{1, 4\}}$.
\medbreak

\subsubsection{Operad $\Mag^{\{2, 4\}}$}
The reader can check that the rewrite rules $\Afr_2 \Rew \Afr_4$ and
$\Afr_4 \Rew' \Afr_2$ are both convergent presentations of
$\Mag^{\{2, 4\}}$.
Thanks to~\cite{Gir18}, the Hilbert series of $\Mag^{\{2, 4\}}$
is equal to~\eqref{equ:Hilbert_series_Mag_1_2}. Thus, for $n \geq 2$,
$\# \Mag^{\{2, 4\}}(n)$ is equal to~\eqref{prop:Close_formula_Mag_1_2}.
\medbreak

Like in Section~\ref{subsubsec:Mag_1_2}, we choose a combinatorial
realization based on integer compositions. Given an integer $i \geq 1$,
we define the binary operation
\begin{math}
    \circ_i^{(2,4)} : \Compositions(n) \times \Compositions(m)
    \to\Compositions(n + m - 1)
\end{math}
for any integer compositions
$\LambdaB := \left(\LambdaB_1, \dots, \LambdaB_p\right)$ and
$\MuB := \left(\MuB_1, \dots, \MuB_q\right)$ of respective arities
$n \geq i$ and $m \geq 1$ by
\begin{equation}
    \LambdaB \circ_i^{(2,4)} \MuB :=
    \begin{cases}
        \left(\LambdaB_1, \dots, \LambdaB_k,
        \MuB_1, \dots, \MuB_{q - 1}, \MuB_q + \LambdaB_{k + 1},
        \LambdaB_{k + 2}, \dots, \LambdaB_p\right)
            & \mbox{if } i = s_{1, k}(\LambdaB), \\
        \left(\LambdaB_1, \dots, \LambdaB_k,
        i - s_{1, k}(\LambdaB), \MuB_1, \dots, \MuB_{q - 1},
        \MuB_q + s_{1, k+1}(\LambdaB) - i, \LambdaB_{k + 2}, \dots,
        \LambdaB_p
        \right)
            & \mbox{otherwise},
    \end{cases}
\end{equation}
where $k \geq 0$ is such that
\begin{math}
    s_{1, k}(\LambdaB) \leq i < s_{1, k+1}(\LambdaB).
\end{math}
\medbreak

\begin{Proposition} \label{prop:Realisation_Mag_2_4}
    The operads $\left(\Compositions, \circ_i^{(2,4)}\right)$
    and $\left(\Compositions, \bar{\circ}_i^{(2,4)}\right)$
    are combinatorial realizations of~$\Mag^{\{2, 4\}}$.
\end{Proposition}
\begin{proof}
    The proof is similar to the one of
    Proposition~\ref{prop:Realisation_Mag_1_2}, thus we just give an
    outline of it.
    We have to show that the operad
    $\left(\Compositions, \circ_i^{(2,4)}\right)$ and $\Mag^{\{2, 4\}}$
    are isomorphic. The set of normal forms of arity $n$ for the
    rewrite rule $\RewContext$ induced by $\Afr_2 \Rew \Afr_4$
    is~\eqref{equ:Normal_forms_A2}. Thus, it is possible to show that
    the map~\eqref{equ:Morphism_A2} is an operad isomorphism from
    $\Mag^{\{2, 4\}}$ to~$\left(\Compositions, \circ_i^{(2,4)}\right)$.
\end{proof}
\medbreak

\subsubsection{Non-isomorphism of the operads}
As shown in the previous sections, the operads of the four considered
equivalence classes $\left\{\Mag^{\{1, 2\}}, \Mag^{\{4, 5\}}\right\}$,
$\left\{\Mag^{\{1, 3\}}, \Mag^{\{3, 5\}}\right\}$,
$\left\{\Mag^{\{1, 4\}}, \Mag^{\{2, 5\}}\right\}$,
and $\left\{\Mag^{\{2, 4\}}\right\}$ have the same Hilbert series. Even
if they can be realized on the same set of integer compositions, all
these operads are pairwise non-isomorphic (and also
non-anti-isomorphic). Indeed, any (anti-)isomorphism between two of
these operads necessarily maps the generator of the first to the
generator of the second, and since by definition the nontrivial
relations between the generators are different from one operad to
another, the operads cannot by (anti-)isomorphic. This remark is also
valid for the corresponding linear operads.
\medbreak

\subsection{Quotients with complicated presentations}
We do not find finite convergent presentations for the operads
$\Mag^{\{2, 3\}}$ and $\Mag^{\{3, 4\}}$. However, thanks to computer
explorations, we conjecture that the rewrite rules

\begin{minipage}{5cm}
\begin{equation}
    \Afr_4 \enspace \Rew \enspace \Afr_3,
\end{equation}
\end{minipage}
\begin{minipage}{9cm}
\begin{equation}
    \begin{tikzpicture}[xscale=.2,yscale=.17,Centering]
        \node(0)at(0.00,-8.25){};
        \node(10)at(10.00,-8.25){};
        \node(2)at(2.00,-8.25){};
        \node(4)at(4.00,-5.50){};
        \node(6)at(6.00,-5.50){};
        \node(8)at(8.00,-8.25){};
        \node[NodeST](1)at(1.00,-5.50){\begin{math}\Product\end{math}};
        \node[NodeST](3)at(3.00,-2.75){\begin{math}\Product\end{math}};
        \node[NodeST](5)at(5.00,0.00){\begin{math}\Product\end{math}};
        \node[NodeST](7)at(7.00,-2.75){\begin{math}\Product\end{math}};
        \node[NodeST](9)at(9.00,-5.50){\begin{math}\Product\end{math}};
        \draw[Edge](0)--(1);
        \draw[Edge](1)--(3);
        \draw[Edge](10)--(9);
        \draw[Edge](2)--(1);
        \draw[Edge](3)--(5);
        \draw[Edge](4)--(3);
        \draw[Edge](6)--(7);
        \draw[Edge](7)--(5);
        \draw[Edge](8)--(9);
        \draw[Edge](9)--(7);
        \node(r)at(5.00,2.06){};
        \draw[Edge](r)--(5);
    \end{tikzpicture}
    \enspace \Rew \enspace
    \begin{tikzpicture}[xscale=.2,yscale=.17,Centering]
        \node(0)at(0.00,-5.50){};
        \node(10)at(10.00,-8.25){};
        \node(2)at(2.00,-8.25){};
        \node(4)at(4.00,-8.25){};
        \node(6)at(6.00,-5.50){};
        \node(8)at(8.00,-8.25){};
        \node[NodeST](1)at(1.00,-2.75){\begin{math}\Product\end{math}};
        \node[NodeST](3)at(3.00,-5.50){\begin{math}\Product\end{math}};
        \node[NodeST](5)at(5.00,0.00){\begin{math}\Product\end{math}};
        \node[NodeST](7)at(7.00,-2.75){\begin{math}\Product\end{math}};
        \node[NodeST](9)at(9.00,-5.50){\begin{math}\Product\end{math}};
        \draw[Edge](0)--(1);
        \draw[Edge](1)--(5);
        \draw[Edge](10)--(9);
        \draw[Edge](2)--(3);
        \draw[Edge](3)--(1);
        \draw[Edge](4)--(3);
        \draw[Edge](6)--(7);
        \draw[Edge](7)--(5);
        \draw[Edge](8)--(9);
        \draw[Edge](9)--(7);
        \node(r)at(5.00,2.06){};
        \draw[Edge](r)--(5);
    \end{tikzpicture},
    \end{equation}
\end{minipage}

\begin{minipage}{6cm}
\begin{equation}
    \begin{tikzpicture}[xscale=.17,yscale=.17,Centering]
        \node(0)at(0.00,-10.40){};
        \node(10)at(10.00,-5.20){};
        \node(12)at(12.00,-5.20){};
        \node(2)at(2.00,-10.40){};
        \node(4)at(4.00,-7.80){};
        \node(6)at(6.00,-7.80){};
        \node(8)at(8.00,-7.80){};
        \node[NodeST](1)at(1.00,-7.80){\begin{math}\Product\end{math}};
        \node[NodeST](11)at(11.00,-2.60){\begin{math}\Product\end{math}};
        \node[NodeST](3)at(3.00,-5.20){\begin{math}\Product\end{math}};
        \node[NodeST](5)at(5.00,-2.60){\begin{math}\Product\end{math}};
        \node[NodeST](7)at(7.00,-5.20){\begin{math}\Product\end{math}};
        \node[NodeST](9)at(9.00,0.00){\begin{math}\Product\end{math}};
        \draw[Edge](0)--(1);
        \draw[Edge](1)--(3);
        \draw[Edge](10)--(11);
        \draw[Edge](11)--(9);
        \draw[Edge](12)--(11);
        \draw[Edge](2)--(1);
        \draw[Edge](3)--(5);
        \draw[Edge](4)--(3);
        \draw[Edge](5)--(9);
        \draw[Edge](6)--(7);
        \draw[Edge](7)--(5);
        \draw[Edge](8)--(7);
        \node(r)at(9.00,1.95){};
        \draw[Edge](r)--(9);
    \end{tikzpicture}
    \enspace \Rew \enspace
    \begin{tikzpicture}[xscale=.17,yscale=.17,Centering]
        \node(0)at(0.00,-7.80){};
        \node(10)at(10.00,-5.20){};
        \node(12)at(12.00,-5.20){};
        \node(2)at(2.00,-7.80){};
        \node(4)at(4.00,-7.80){};
        \node(6)at(6.00,-10.40){};
        \node(8)at(8.00,-10.40){};
        \node[NodeST](1)at(1.00,-5.20){\begin{math}\Product\end{math}};
        \node[NodeST](11)at(11.00,-2.60)
            {\begin{math}\Product\end{math}};
        \node[NodeST](3)at(3.00,-2.60){\begin{math}\Product\end{math}};
        \node[NodeST](5)at(5.00,-5.20){\begin{math}\Product\end{math}};
        \node[NodeST](7)at(7.00,-7.80){\begin{math}\Product\end{math}};
        \node[NodeST](9)at(9.00,0.00){\begin{math}\Product\end{math}};
        \draw[Edge](0)--(1);
        \draw[Edge](1)--(3);
        \draw[Edge](10)--(11);
        \draw[Edge](11)--(9);
        \draw[Edge](12)--(11);
        \draw[Edge](2)--(1);
        \draw[Edge](3)--(9);
        \draw[Edge](4)--(5);
        \draw[Edge](5)--(3);
        \draw[Edge](6)--(7);
        \draw[Edge](7)--(5);
        \draw[Edge](8)--(7);
        \node(r)at(9.00,1.95){};
        \draw[Edge](r)--(9);
    \end{tikzpicture},
    \end{equation}
\end{minipage}
\begin{minipage}{8cm}
\begin{equation}
    \begin{tikzpicture}[xscale=.2,yscale=.18,Centering]
        \node(0)at(0.00,-8.67){};
        \node(10)at(10.00,-4.33){};
        \node(12)at(12.00,-4.33){};
        \node(2)at(2.00,-10.83){};
        \node(4)at(4.00,-10.83){};
        \node(6)at(6.00,-6.50){};
        \node(8)at(8.00,-4.33){};
        \node[NodeST](1)at(1.00,-6.50){\begin{math}\Product\end{math}};
        \node[NodeST](11)at(11.00,-2.17){\begin{math}\Product\end{math}};
        \node[NodeST](3)at(3.00,-8.67){\begin{math}\Product\end{math}};
        \node[NodeST](5)at(4.00,-5.33){\begin{math}\Product\end{math}};
        \node[NodeST](7)at(7.00,-2.17){\begin{math}\Product\end{math}};
        \node[NodeST](9)at(9.00,0.00){\begin{math}\Product\end{math}};
        \draw[Edge](0)--(1);
        \draw[Edge](1)--(5);
        \draw[Edge](10)--(11);
        \draw[Edge](11)--(9);
        \draw[Edge](12)--(11);
        \draw[Edge](2)--(3);
        \draw[Edge](3)--(1);
        \draw[Edge](4)--(3);
        \draw[Edge,dotted](5)edge[]node[font=\footnotesize]{
            \begin{math}k\end{math}\hspace*{.5cm}}(7);
        \draw[Edge](6)--(5);
        \draw[Edge](7)--(9);
        \draw[Edge](8)--(7);
        \node(r)at(9.00,1.62){};
        \draw[Edge](r)--(9);
    \end{tikzpicture}
    \enspace \Rew \enspace
    \begin{tikzpicture}[xscale=.19,yscale=.18,Centering]
        \node(0)at(0.00,-7.80){};
        \node(10)at(10.00,-5.20){};
        \node(12)at(12.00,-5.20){};
        \node(2)at(2.00,-7.80){};
        \node(4)at(4.00,-7.80){};
        \node(6)at(6.00,-10.40){};
        \node(8)at(8.00,-10.40){};
        \node[NodeST](1)at(1.00,-5.20){\begin{math}\Product\end{math}};
        \node[NodeST](11)at(11.00,-2.60){\begin{math}\Product\end{math}};
        \node[NodeST](3)at(3.00,-2.60){\begin{math}\Product\end{math}};
        \node[NodeST](5)at(5.00,-5.20){\begin{math}\Product\end{math}};
        \node[NodeST](7)at(7.00,-7.80){\begin{math}\Product\end{math}};
        \node[NodeST](9)at(9.00,0.00){\begin{math}\Product\end{math}};
        \draw[Edge](0)--(1);
        \draw[Edge](1)--(3);
        \draw[Edge](10)--(11);
        \draw[Edge](11)--(9);
        \draw[Edge](12)--(11);
        \draw[Edge](2)--(1);
        \draw[Edge](3)--(9);
        \draw[Edge](4)--(5);
        \draw[Edge](5)--(3);
        \draw[Edge](6)--(7);
        \draw[Edge,dotted](7)edge[]node[font=\footnotesize]{
            \hspace*{.4cm}\begin{math}k\end{math}}(5);
        \draw[Edge](8)--(7);
        \node(r)at(9.00,1.95){};
        \draw[Edge](r)--(9);
    \end{tikzpicture},
    \enspace k \geq 1
\end{equation}
\end{minipage}

\noindent form a convergent presentation of~$\Mag^{\{3, 4\}}$.
We checked that presentation until arity $40$. From this
rewrite relation $\Rew$, we also conjecture that the Hilbert series of
$\Mag^{\{2, 3\}}$ and, by Lemma~\ref{lem:recall_anti_isomorphic}, of
$\Mag^{\{3, 4\}}$ are
\begin{equation}
    \HilbertSeries_{\Mag^{\{2, 3\}}}(t) =
    \HilbertSeries_{\Mag^{\{3, 4\}}}(t) =
    \frac{t}{(1 - t)^3}
    \left(1 - 2t + 2t^2 + t^4 - t^6\right).
\end{equation}
By Taylor expansion, we have the sequence
\begin{equation}
    1, 1, 2, 4, 8, 14, 21, 29, 38, 48
\end{equation}
for the first dimensions of $\Mag^{\{2, 3\}}$ and $\Mag^{\{3, 4\}}$.
For $n \geq 5$,
\begin{equation}
    \#\Mag^{\{2, 3\}}(n) =\#\Mag^{\{3, 4\}}(n) = \frac{n(n+1)}{2} - 7.
\end{equation}
\medbreak

\section*{Conclusion and perspectives}
In this paper, we have considered some quotients of the magmatic operad
in both the linear and the set-theoretic frameworks. We focused mainly
our study on comb associative operads and collected properties by
using computer exploration and rewrite systems on trees. There are many
ways to extend this work. Here follow some few further research
directions.
\medbreak

A first research direction consists in finding convergent presentations
for (all or most of) the operads $\CAs{\gamma}$ in
order to describe algebraic and combinatorial properties of them (as
describing explicit bases, computing Hilbert series, and providing
combinatorial realizations). This has been done only in the case
$\gamma = 3$. For some other cases, we only have conjectural and
experimental data (see
Section~\ref{sec:higher_comb_associative_operads}).
\medbreak

Following ideas existing for word rewriting theory~\cite{GGM15}, a
second axis consists in allowing new generators for the operads
$\CAs{\gamma}$ in order to obtain finite convergent presentations
when $\gamma \geq 4$. Indeed, the Buchberger semi-algorithm works by
adding rewrite rules to a set of rewrite rules to obtain a convergent
rewrite system. An orthogonal procedure consists rather in adding new
generators (new labels for internal nodes in the trees) in order to
obtain convergent rewrite systems. More generally, we also would like to
use these ideas for other magmatic quotients, such as the operad
$\Mag^{\{3, 4\}}$ that we did not study entirely in
Section~\ref{sec:MAg_3}.
\medbreak

A third axis consists in studying if the completion of presentations of
quotients of the magmatic operad maintains links with
the lattice structure introduced in Section~\ref{sec:Magmatic_operads}.
More precisely, assuming that we have completed the presentations of
the quotients $\Oca_1$ and $\Oca_2$ of $\Mag$,
as well as the one of the lower-bound $\Oca_1 \InfQMag \Oca_2$, the
question consists in designing an algorithm for computing a completion
of a presentation of the upper-bound
$\Oca_1 \SupQMag \Oca_2$. Of course, the same question also makes sense
for the lattice of comb associative operads introduced in
Section~\ref{sec:CAs_d}.
\medbreak

Let us address now a perspective fitting more in a combinatorial
context. As mentioned in the introduction of this article, we suspect
that some combinatorial properties of quotients $\Mag/_{\Congr}$ of
$\Mag$ derive from properties of the equivalence relation generating the
operad congruence $\Congr$. More precisely, we would like to investigate
if, when this equivalence relation is a set of Tamari intervals (or is
closed by interval, or satisfies some other classical properties coming
from poset theory), one harvests a nice description of the Hilbert
series and of a combinatorial realization of $\Mag/_{\Congr}$.
\medbreak

A last research axis relies on the study on the $2$-magmatic operad
$\TwoMag$, that is, the free operad generated by two binary elements.
The analog of the associative operad in this context is the operad
$\TwoAs$~\cite{LR06} defined as the quotient of $\TwoMag$ by the
congruence saying that the two generators are associative. This operad
has a nice combinatorial realization in terms of alternating bicolored
Schröder trees. The question consists here in generalizing our main
results for the quotients of $\TwoMag$ and the generalizations of
$\TwoAs$ (that is, the definition of analogs of comb associative operads
and the study of their presentations).
\medbreak

\bibliographystyle{alpha}
\bibliography{Bibliography}

\begin{thebibliography}{GGM15}

\bibitem[Ani86]{Ani86}
D.~J. Anick.
\newblock {On the homology of associative algebras}.
\newblock {\em Trans. Amer. Math. Soc.}, 296(2):641--659, 1986.

\bibitem[AVL62]{AVL62}
G.M. Adelson-Velsky and E.~M. Landis.
\newblock {An algorithm for the organization of information}.
\newblock {\em Soviet Mathematics Doklady}, 3:1259--1263, 1962.

\bibitem[BN98]{BN98}
F.~Baader and T.~Nipkow.
\newblock {\em {Term rewriting and all that}}.
\newblock Cambridge University Press, 1998.

\bibitem[CCG18]{CCG18}
C.~Chenavier, C.~Cordero, and S.~Giraudo.
\newblock {Generalizations of the associative operad and convergent rewrite
  systems}.
\newblock {\em Higher-Dimensional Rewriting and Algebra}, Preprint no. 143,
  2018.

\bibitem[Cha06]{Cha06}
F.~Chapoton.
\newblock {Sur le nombre d'intervalles dans les treillis de Tamari}.
\newblock {\em Sém. Lothar. Combin.}, 55, 2006.

\bibitem[DK10]{DK10}
V.~Dotsenko and A.~Khoroshkin.
\newblock Gr\"obner bases for operads.
\newblock {\em Duke Math. J.}, 153(2):363--396, 2010.

\bibitem[GGM15]{GGM15}
S.~Gaussent, Y.~Guiraud, and P.~Malbos.
\newblock {Coherent presentations of Artin monoids}.
\newblock {\em Compos. Math.}, 151(5):957--998, 2015.

\bibitem[Gir16]{Gir16}
S.~Giraudo.
\newblock {Operads from posets and Koszul duality}.
\newblock {\em Eur. J. Combin.}, 56C:1--32, 2016.

\bibitem[Gir18]{Gir18}
S.~Giraudo.
\newblock {Tree series and pattern avoidance in syntax trees}.
\newblock {\em Prepublication}, 2018.

\bibitem[GK94]{GK94}
V.~Ginzburg and M.~M. Kapranov.
\newblock {Koszul duality for operads}.
\newblock {\em Duke Math. J.}, 76(1):203--272, 1994.

\bibitem[Hof10]{Hof10}
E.~Hoffbeck.
\newblock {A Poincar\'e-Birkhoff-Witt criterion for Koszul operads}.
\newblock {\em Manuscripta Math.}, 131(1-2):87--110, 2010.

\bibitem[HT72]{HT72}
S.~Huang and D.~Tamari.
\newblock {Problems of associativity: A simple proof for the lattice property
  of systems ordered by a semi-associative law}.
\newblock {\em J. Comb. Theory. A}, 13:7--13, 1972.

\bibitem[Knu98]{Knu98}
D.~Knuth.
\newblock {\em The art of computer programming, volume 3: Sorting and
  searching}.
\newblock Addison Wesley Longman, 1998.

\bibitem[KP15]{KP15}
A.~Khoroshkin and D.~Piontkovski.
\newblock {On generating series of finitely presented operads}.
\newblock {\em J. Algebra}, 426:377--429, 2015.

\bibitem[Lan02]{Lan02}
S.~Lang.
\newblock {\em {Algebra}}.
\newblock Springer-Verlag New York Inc., 3rd edition, 2002.

\bibitem[LR06]{LR06}
J.-L. Loday and M.~Ronco.
\newblock {On the structure of cofree Hopf algebras}.
\newblock {\em J. Reine Angew. Math.}, 592:123--155, 2006.

\bibitem[LV12]{LV12}
J.-L. Loday and B.~Vallette.
\newblock {\em {Algebraic Operads}}, volume 346 of {\em Grundlehren der
  mathematischen Wissenschaften}.
\newblock Springer, Heidelberg, 2012.
\newblock Pages xxiv+634.

\bibitem[New42]{New42}
M.~H.~A. Newman.
\newblock {On theories with a combinatorial definition of ``equivalence.''}.
\newblock {\em Ann. Math.}, 43(2):223--243, 1942.

\bibitem[Pri70]{Pri70}
S.~B. Priddy.
\newblock {Koszul resolutions}.
\newblock {\em Trans. Amer. Math. Soc.}, 152:39--60, 1970.

\bibitem[Row10]{Row10}
E.~S. Rowland.
\newblock {Pattern avoidance in binary trees}.
\newblock {\em J. Comb. Theory A}, 117(6):741--758, 2010.

\bibitem[Tam62]{Tam62}
D.~Tamari.
\newblock {The algebra of bracketings and their enumeration}.
\newblock {\em Nieuw Arch. Wisk.}, 10(3):131--146, 1962.

\bibitem[Zin12]{Zin12}
G.~W. Zinbiel.
\newblock {Encyclopedia of types of algebras 2010}.
\newblock In {\em Operads and universal algebra}, volume~9 of {\em Nankai Ser.
  Pure Appl. Math. Theoret. Phys.}, pages 217--297. World Sci. Publ.,
  Hackensack, NJ, 2012.

\end{thebibliography}

\end{document}